\documentclass[AMA,STIX1COL]{WileyNJD-v2}

\usepackage{graphicx}
\usepackage{subfig}

\def\N{\mathbb N}

\def\del{\Delta}
\def\gam{\gamma}
\def\eps{\varepsilon}
\def\ba{\begin{align}}

\articletype{Article Type}%

\received{26 April 2016}
\revised{6 June 2016}
\accepted{6 June 2016}

\begin{document}

\title{Delay Compensated Control of the Stefan Problem and Robustness to Delay Mismatch\protect}

\author[1]{Shumon Koga*}

\author[2]{Delphine Bresch-Pietri}

\author[1]{Miroslav Krstic}

\authormark{SHUMON KOGA \textsc{et al}}

\address[1]{\orgdiv{Department of Mechanical and Aerospace Engineering}, \orgname{University of California at San Diego}, \orgaddress{\state{La Jolla, CA 92093-0411}, \country{USA}}}

\address[2]{\orgdiv{CAS - Centre Automatique et Syst{\`e}mes}, \orgname{MINES ParisTech, PSL Research University}, \orgaddress{\state{Paris}, \country{France}}}

\corres{*Correspondence to: Shumon Koga, Department of Mechanical and Aerospace Engineering, University of California at San Diego, La Jolla, CA 92093-0411, USA. \email{skoga@eng.ucsd.edu}}


\abstract[Summary]{This paper presents a control design for the one-phase Stefan problem under actuator delay via a backstepping method. The Stefan problem represents a liquid-solid phase change phenomenon which describes the time evolution of a material's temperature profile and the interface position. The actuator delay is modeled by a first-order hyperbolic partial differential equation (PDE), resulting in a cascaded transport-diffusion PDE system defined on a time-varying spatial domain described by an ordinary differential equation (ODE). Two nonlinear backstepping transformations are utilized for the control design. The setpoint restriction is given to guarantee a physical constraint on the proposed controller for the melting process. This constraint ensures the exponential convergence of the moving interface to a setpoint and the exponential stability of the temperature equilibrium profile and the delayed controller in the ${\cal H}_1$ norm. Furthermore, robustness analysis with respect to the delay mismatch between the plant and the controller is studied, which provides analogous results to the exact compensation by restricting the control gain.}

\keywords{Stefan problem, distributed parameter systems, time delay systems, backstepping, nonlinear stabilization}

\jnlcitation{\cname{%
\author{Williams K.}, 
\author{B. Hoskins}, 
\author{R. Lee}, 
\author{G. Masato}, and 
\author{T. Woollings}} (\cyear{2016}), 
\ctitle{A regime analysis of Atlantic winter jet variability applied to evaluate HadGEM3-GC2}, \cjournal{Q.J.R. Meteorol. Soc.}, \cvol{2017;00:1--6}.}

\maketitle


\section{INTRODUCTION}
\subsection{Background} 

Liquid-solid phase transitions are physical phenomena which appear in various kinds of science and engineering processes. Representative applications include sea-ice melting and freezing~\cite{Shumon17seaice}, continuous casting of steel \cite{petrus2012}, cancer treatment by cryosurgeries \cite{Rabin1998}, additive manufacturing for materials of both polymer \cite{koga2018polymer} and metal \cite{chung2004}, crystal growth~\cite{conrad_90}, lithium-ion batteries \cite{koga2017battery}, and thermal energy storage systems \cite{zalba03}. Physically, these processes are described by a temperature profile along a liquid-solid material, where the dynamics of the liquid-solid interface is influenced by the heat flux induced by melting or solidification. A mathematical model of such a physical process is called the Stefan problem\cite{Gupta03}, which is formulated by a diffusion PDE defined on a time-varying spatial domain. The domain's length dynamics is described by an ODE dependent on the Neumann boundary value of the PDE state. Apart from the thermodynamical model, the Stefan problem has been employed to model several chemical, electrical, and social dynamics such as tumor growth process \cite{Friedman1999}, domain walls in ferroelectric thin films \cite{mcgilly2015}, spreading of invasive species in ecology \cite{Du2010speading}, and information diffusion on social networks \cite{Lei2013}.

While the numerical analysis of  the one-phase Stefan problem is broadly covered in the literature, their control related problems have been addressed relatively fewer. In addition to it, most of the proposed control approaches are based on finite-dimensional approximations with the assumption of  an explicitly given moving boundary dynamics \cite{Daraoui2010},\cite{Armaou01},\cite{Petit10}. Diffusion-reaction processes with an explicitly known moving boundary dynamics are investigated in  \cite{Armaou01} based on the concept of inertial manifold \cite{Christofides98_Parabolic} and the  partitioning of the infinite dimensional dynamics into slow and fast  finite dimensional modes. Motion planning boundary control has been adopted in   \cite{Petit10} to ensure asymptotic stability of a one-dimensional one-phase nonlinear Stefan problem   assuming  a prior known  moving boundary  and deriving the manipulated input  from the solutions of   the inverse problem. However, the series representation introduced in \cite{Petit10}  leads to highly complex solutions that reduce controller design possibilities. 
 
For control objectives, infinite-dimensional approaches have been used for stabilization of  the temperature profile and the moving interface of a 1D Stefan problem, such as enthalpy-based feedback~\cite{petrus2012} and geometric control~\cite{maidi2014}. These works designed control laws ensuring the asymptotical stability of the closed-loop system in the ${ L}_2$ norm. However, the  result in \cite{maidi2014}  is stated based on physical assumptions on the liquid temperature being greater than the melting point, which needs to be guaranteed by showing strictly positive boundary input. 

Recently, boundary feedback controllers for the Stefan problem have been designed via a ``backstepping transformation" \cite{krstic2008boundary,andrew2004} which has been used for many other classes of infinite-dimensional systems. For instance, \cite{Shumon16} designed a state feedback control law by introducing a nonlinear backstepping transformation for moving boundary PDE, which achieved the exponentially stabilization of the closed-loop system in the ${\cal H}_1$ norm without imposing any {\em a priori} assumption. Based on the technique, \cite{Shumon16CDC} designed an observer-based output feedback control law for the Stefan problem, \cite{Shumon18journal} extended the results in \cite{Shumon16, Shumon16CDC} by studying the robustness with respect to the physical parameters and developed an analogous design with Dirichlet boundary actuation, \cite{Shumon17ACC} designed a state feedback control for the Stefan problem under the material's convection, and \cite{koga2018ISS} investigated an input-to-state stability of the control of Stefan problem in \cite{Shumon16} with respect to an unknown heat loss at the interface. 

In the presence of actuator delay, a delay compensation technique has been developed intensively for many classes of systems using a backstepping transformation \cite{krstic2009delay}: see \cite{krstic2008backstepping} for linear ODE systems and \cite{krstic2010} for nonlinear ODE systems. Using the Lyapunov method, \cite{krstic2008lyapunov} presented the several analysis of the predictor-based feedback control for ODEs such as robustness with respect to the delay mismatch and disturbance attenuation. To deal with systems under unknown and arbitrary large actuator delay, a Lyapunov-based delay-adaptive control design was developed in \cite{bresch2010, bresch2014} for both linear and nonlinear ODEs with certain systems, and \cite{bresch2009} extended the design for trajectory tracking of uncertain linear ODEs. For control of unstable parabolic PDE under a long input delay, \cite{krstic2009} designed the stabilizing controller by introducing two backstepping transformations for the stabilization of the unstable PDE and the compensation of the delay. By the similar technique, in \cite{Tang2011delay} the coupled diffusion PDE-ODE system in the presence of the actuator delay is stabilized. Implementation issues on the predictor-bsaed feedback are covered in \cite{karafyllis2017book} by studying the closed-loop analysis under the sampled-data control. 

\subsection{Results and contributions} 

Our conference paper \cite{koga2017CDC} presented the delay compensated control for the one-phase Stefan problem under actuator delay for the stabilization of the interface position and the temperature profile at a desired setpoint and the equillibrium temperature. This paper extends the results in \cite{koga2017CDC} by: 

\begin{itemize}
\item proving that the designed controller is equivalent to the prediction of the nominal control law for delay-free Stefan problem over a time interval corresponding to the input delay, 
\item and addressing the robustness analysis of the closed-loop system with respect to the mismatch between the delay in the plant and the one compensated by the designed control. 
\end{itemize}

First, combining our previous result in \cite{Shumon16} with the result in \cite{krstic2009}, two nonlinear backstepping transformations for moving boundary PDE are employed. One is for the delay-free control design of Stefan problem based on \cite{Shumon16}, and the other is for the compensation of actuator delay formulated with Volterra and Fredholm type transformations based on \cite{krstic2009}. The associated boundary feedback controller remains positive under a setpoint restriction due to the energy conservation, which guarantees a condition of the model to be valid. The closed-loop system with the proposed delay compensated controller achieves the exponential stabilization of the moving interface to the desired setpoint while ensuring the  exponential stability of the temperature profile and the controller to the equillibrium set in the ${\cal H}_{1}$-norm sense. Furthermore, the robustness analysis is investigated by proving that the positivity of the controller and the exponential stability of the closed-loop systems hold for a given delay mismatch under sufficiently small control gain. 


\subsection{Organizations} 

This paper is organized as follows. The Stefan problem with actuator delay is presented  in Section \ref{model}, and the control objective and our main result are stated in Section \ref{statement}. Section \ref{nonlineartarget} introduces a backstepping transformation for moving boundary problems which enables us to design the state feedback control law. The physical constraints of this problem are stated  in Section \ref{sec:constraints}. The stability analysis of the closed-loop system is established in Section \ref{stability}. The equivalence of the designed control with a prediction of the nominal control law is shown in Section \ref{sec:prediction}. Robustness analysis with respect to the delay mismatch is studied in Section \ref{sec:robust}. Supportive numerical simulations are provided in Section \ref{simulation}. The paper ends with the conclusion in  Section \ref{conclusion}. 

\section{Description of the Physical Process}\label{model}
Consider a physical model which describes the melting or solidification mechanism in a pure one-component material of length $L$ in one dimension. In order to mathematically describe the position at which phase transition from liquid to solid occurs, we divide the domain $[0, L]$ into two time-varying sub-domains, namely,  the interval $[0,s(t)]$ which contains the liquid phase, and the interval $[s(t),L]$ that contains the solid  phase. A heat flux enters the material through the boundary at $x=0$ (the external boundary of the liquid phase) which affects the dynamics of the liquid-solid interface. The boundary heat flux is manipulated as a controller, and here we impose an actuator delay which is caused by several reasons such as computational time or communication delay. As a consequence, the heat equation alone does not provide a complete description of the phase transition and must be coupled with the dynamics that describes the moving boundary. This configuration is shown in Fig.~\ref{fig:stefan}.

\begin{figure}[t]
\centering
\includegraphics[width=3.2in]{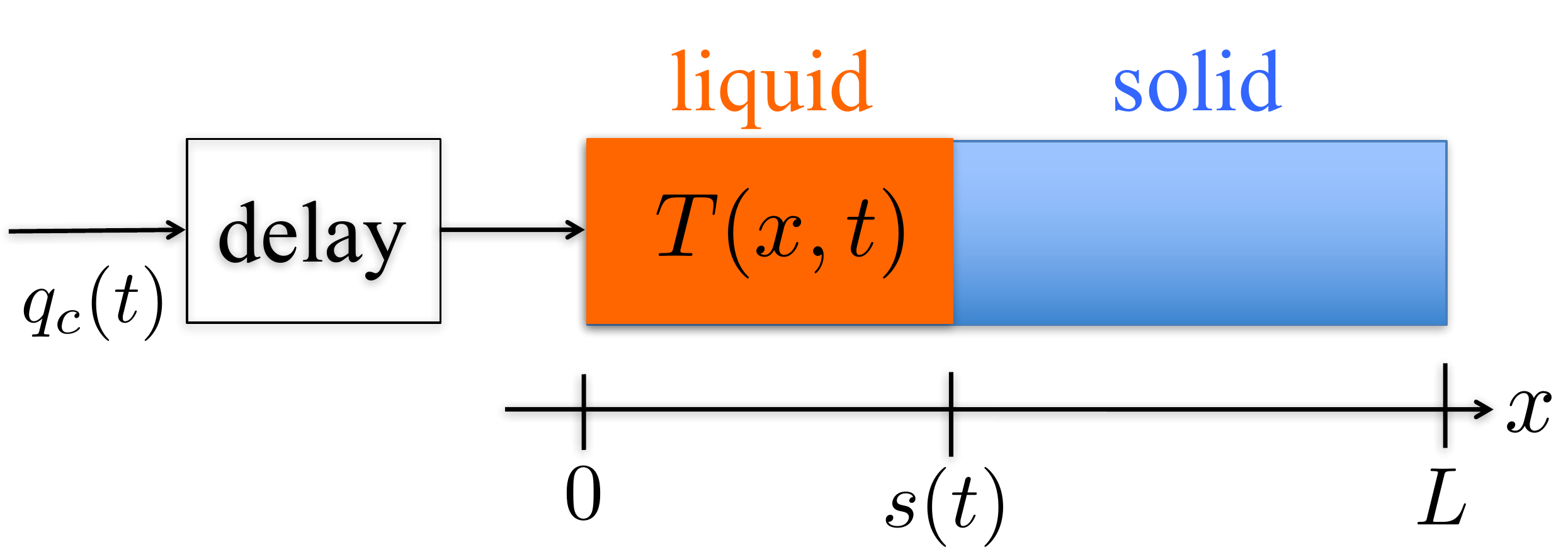}
\caption{Schematic of 1D Stefan problem with actuator delay.}
\label{fig:stefan}
\end{figure}

Assuming that the temperature in the liquid phase is not lower than the melting temperature $T_{{\rm m}}$ of the material, the following coupled system can be derived.
\begin{itemize} 
\item The diffusion equation of the temperature in the liquid-phase is described by
\begin{align}\label{eq:stefanPDE}
T_t(x,t)&=\alpha T_{xx}(x,t), \hspace{1mm} 0\leq x\leq s(t),\quad  \alpha :=\frac{k}{\rho C_p}, 
\end{align}
with the boundary conditions
\begin{align}\label{eq:stefancontrol}
-k T_x(0,t)&=q_{{\rm c}}(t-D), \\ \label{eq:stefanBC}
T(s(t),t)&=T_{{\rm m}},
\end{align}
and the initial values
\begin{align}\label{eq:stefanIC}
T(x,0)=T_0(x), \quad s(0) = s_0, 
\end{align}
where $T(x,t)$,  $q_{{\rm c}}(t)$,  $\rho$, $C_p$, $k$, and $D$ are the distributed temperature of the liquid phase, manipulated heat flux, liquid density, the liquid heat capacity, the liquid heat conductivity, and the input time delay respectively.
\item The local  energy balance at the position of the liquid-solid interface $x=s(t)$ leads to the Stefan condition defined as  the following ODE
\begin{align}\label{eq:stefanODE}
 \dot{s}(t)=-\beta T_x(s(t),t), \quad \beta :=\frac{k}{\rho \Delta H^*}, 
\end{align}
where $\Delta H^*$ denotes the latent heat of fusion. Equation \eqref{eq:stefanODE}   expresses the velocity of the liquid-solid moving interface. 
\end{itemize}
For the sake of brevity, we refer the readers to  \cite{Gupta03}, where the Stefan condition of a solidification process is derived. 
\begin{remark}As the moving interface  $s(t)$ depends on the temperature, the problem defined in  \eqref{eq:stefanPDE}--\eqref{eq:stefanODE}  is nonlinear.\end{remark}
\begin{remark}\label{assumption}
Due to the so-called isothermal interface condition
 that prescribes the melting temperature $T_{{\mathrm m}}$ at the interface through  \eqref{eq:stefanBC},
this form of the Stefan problem is a reasonable model only if the following condition holds:
\begin{align}
\label{valid1}T(x,t) \geq& T_{{\mathrm m}} \quad \textrm{ for }\quad \forall x \in [0,s(t)], \quad \forall t>0. 
\end{align}
\end{remark}

The model validity requires the liquid temperature to be greater than the melting temperature and such a condition  yields the following property on moving interface. 
\begin{lemma}\label{monoinc}
If the model validity condition \eqref{valid1} holds, then the moving interface is always increasing, i.e. 
\begin{align}
\label{valid2}\dot{s}(t)\geq&0, \quad \textrm{ for }\quad \forall t\geq0. 
\end{align}
\end{lemma}

Applying Hopf's Lemma to the boundary condition \eqref{eq:stefanBC} and the condition \eqref{valid1}, the dynamics \eqref{eq:stefanODE} yields Lemma \ref{monoinc} as shown in \cite{Gupta03}. By Remark \ref{assumption}, it is justified to impose the following assumption on the initial values. 
\begin{assumption}\label{initial} 
$s_0>0$ and there exists Lipschitz constant $ H>0$ such that 
\begin{align}
T_{{\rm m}}\leq T_0(x) \leq T_{{\rm m}} + H(s_0 - x), \quad \forall x\in(0,s_0).
\end{align}
 \end{assumption}
 
 Then, the condition \eqref{valid1} is guaranteed by the following lemma. 
\begin{lemma}\label{lem1}
With Assumption \ref{initial}, the model validity condition \eqref{valid1} remains if
\begin{align}
q_{{\rm c}}(t-D) \geq 0,  \quad \forall t>0. 
\end{align}
\end{lemma}

The proof of Lemma \ref{lem1} is provided by maximum principle as shown in \cite{Gupta03}. Hence, we impose the following assumption. 
\begin{assumption}\label{assum1}
The past input maintains positive, i.e. 
\begin{align}
q_{{\rm c}}(t) \geq 0, \quad -D<\forall t<0. 
\end{align}
\end{assumption}

With Assumption \ref{initial} and \ref{assum1}, the model validity condition \eqref{valid1} remains if $q_{{\rm c}}(t) \geq 0$ for $\forall t>0$ by Lemma \ref{lem1}.  

\section{Control Problem statement}\label{statement}
The control objective is to drive the moving interface $s(t)$ to a desired setpoint $s_{{\rm r}}$ while ensuring the convergence of the ${\cal H}_1$-norm of the temperature $T(x,t)$ in the liquid phase by manipulating the heat flux $q_{{\rm c}}(t)$. Hence, we aim to achieve
\begin{align} \label{objective} 
s(t) \to s_{{\rm r}}, \quad T(x,t) \to T_{{\rm m}}, \quad \textrm{as} \quad  t\to \infty
\end{align} 
for given $(T_{0}(x), s_0)$ which satisfies Assumption \ref{initial}. The condition $q_{{\rm c}}(t)>0$ for $t>0$ imposes the choice of the setpoint as described below. The plant \eqref{eq:stefanPDE}--\eqref{eq:stefanODE} obeys the following energy conservation law:
\begin{align}\label{energy}
\frac{d}{dt} \left(\frac{k}{\alpha} \int_{0}^{s(t)} (T(x,t) - T_{{\rm m}}) dx + \frac{k}{\beta} s(t) + \int_{t-D}^{t} q_{{\rm c}}(\theta) d\theta \right) = q_{{\rm c}}(t).
\end{align}
The left hand side of \eqref{energy} denotes the growth of internal energy of the plant and the stored energy by the delayed heat controller, and its right hand side denotes the external work provided by the injected heat flux. The control objective is achieved if and only if the following limit on the total energy is satisfied:
\begin{align}\label{internal}
\lim_{t\to \infty} \left (\frac{k}{\alpha} \int_{0}^{s(t)} (T(x,t) - T_{{\rm m}}) dx + \frac{k}{\beta} s(t) + \int_{t-D}^{t} q_{{\rm c}}(\theta) d\theta \right) = \frac{k}{\beta} s_{{\rm r}},
\end{align}
which can be derived by substituting \eqref{objective} and $q_{c}(t) \to 0$ into the left hand side of \eqref{internal}. 
Taking integration of \eqref{energy} from $t=0$ to $t=\infty$ with the help of $q_{{\rm c}}(t)>0$ for $t>0$ and \eqref{internal}, the following assumption on the setpoint is provided. 
\begin{assumption}\label{assum2}
The setpoint is chosen to satisfy
\begin{align}\label{setpoint}
s_{{\rm r}} > s_0 + \beta \left(\int_{-D}^{0} \frac{q_{{\rm c}}(t)}{k} dt + \frac{1}{\alpha} \int_{0}^{s_0}  (T_0(x)-T_{{\rm m}}) dx \right). 
\end{align} 
\end{assumption}

Next, we state our main result.
\begin{theorem}\label{thm1}
Under Assumptions \ref{initial}-\ref{assum2}, the closed-loop system consisting of the plant \eqref{eq:stefanPDE}--\eqref{eq:stefanODE} and the control law
\begin{align}\label{delaycomp}
q_{{\rm c}}(t) =& - c  \left(\int_{t-D}^{t} q_{{\rm c}}(\theta) d\theta + \frac{k}{\alpha} \int_{0}^{s(t)}  (T(x,t) - T_{{\rm m}}) dx + \frac{k}{\beta} (s(t) -s_{{\rm r}}) \right), 
\end{align}
where $c>0$ is an arbitral control gain, maintains the model validity \eqref{valid1} and is exponentially stable in the sense of the norm
\begin{align}\label{h1}
 ||T(x,t)-T_{{\rm m}}||^{2}_{{\cal H}_1(0,s(t))}+(s(t)-s_{{\rm r}})^2+ ||q_{{\rm c}}(t-x)||^2_{{\cal H}_1(0,D)}. 
\end{align}
\end{theorem}

The proof of Theorem \ref{thm1} is established through Sections \ref{nonlineartarget}--\ref{stability}.

\section{Backstepping Transformation}\label{nonlineartarget}
\subsection{Change of variables}
Introduce reference error variables defined by
\begin{align}\label{state}
u(x,t) :=& T(x,t) - T_{{\rm m}}, \quad X(t) := s(t) - s_{{\rm r}}. 
\end{align}
Next, we introduce a variable 
\begin{align}\label{vdef}
v(x,t) = \frac{q_{{\rm c}}(t-x-D)}{k}.
\end{align}
Then, \eqref{vdef} gives the boundary values of current input $v(-D,t) =q_{{\rm c}}(t)/k$ and delayed input $v(0,t)= q_{{\rm c}}(t-D)/k$, and $v(x,t)$ satisfies a transport PDE. Hence, the coupled $(v,u,X)$-system is described as
\begin{align}\label{orisys1}
v_t(x,t) =& -v_x(x,t),\quad -D<x<0\\
\label{orisys2}v(-D,t) =& q_{{\rm c}}(t)/k, \\
u_x(0,t) =& -v(0,t),\\
\label{orisys4} u_{t}(x,t) =& \alpha u_{xx}(x,t),  \quad 0<x<s(t)\\
u(s(t),t) =&0,\\
\label{orisys6}\dot{X}(t) =& -\beta u_{x}(s(t),t). 
\end{align}
Now, the control objective is to design $q_{{\rm c}}(t)$ to stabilize the coupled $(v,u,X)$-system at the origin. 
\subsection{Direct transformation}
We consider backstepping transformations for the coupled PDEs-ODE system as
\begin{align}\label{eq:DB1}
w(x,t)=&u(x,t)-\frac{c}{\alpha} \int_{x}^{s(t)} (x-y)u(y,t) dy - \frac{c}{\beta}(x-s(t)) X(t), \\
\label{eq:DB2}z(x,t) =& v(x,t) + c \int_{x}^{0} v(y,t) dy+ \frac{c}{\alpha} \int_{0}^{s(t)}  u(y,t) dy + \frac{c}{\beta} X(t). 
\end{align}
The transformation \eqref{eq:DB1} is the same nonlinear transformation as the one proposed in \cite{Shumon16} for delay-free Stefan problem. The formulation of \eqref{eq:DB2} is motivated by a design in fixed domain introduced in \cite{krstic2009}. Taking derivatives of \eqref{eq:DB1} and \eqref{eq:DB2} in $x$ and $t$ along with the solution of the system \eqref{orisys1}--\eqref{orisys6}, we have 
\begin{align}\label{wxder}
w_{x}(x,t)=&u_{x}(x,t)-\frac{c}{\alpha} \int_{x}^{s(t)} u(y,t) dy - \frac{c}{\beta} X(t), \\
\label{zxder}z_{x}(x,t) =& v_{x}(x,t) - c v(x,t) , \\
\label{ztder}z_{t}(x,t) =& -v_{x}(x,t) - c \int_{x}^{0} v_{y}(y,t) dy + c \int_{0}^{s(t)}  u_{yy}(y,t) dy - c u_{x}(s(t),t), \notag\\
=& -v_{x}(x,t) + c v(x,t) . 
\end{align}
By \eqref{zxder} and \eqref{ztder}, we get $z_t(x,t) = -z_x(x,t)$. In addition, by substituting $x=0$ in \eqref{eq:DB2} and \eqref{wxder}, $w_x(0,t) = -z(0,t)$ holds. On the other hand, because $w$ transformation does not depend on $v$, $w$ system is not changed from the delay-free target system given in \cite{Shumon16}. Thus, the target $(z,w,X)$-system is obtained by
\begin{align}\label{tar1}
z_t(x,t) =& -z_x(x,t),\quad -D<x<0\\
\label{tar2}z(-D,t) =& 0, \\
\label{tar3}w_x(0,t) =& -z(0,t),\\
\label{tar4}w_{t}(x,t) =& \alpha w_{xx}(x,t) + \frac{c}{\beta} \dot{s}(t) X(t),  \quad 0<x<s(t)\\
w(s(t),t) =&0,\\
\label{tarODE}\dot{X}(t) =& -cX(t) -\beta w_{x}(s(t),t). 
\end{align}
The control design is achieved through evaluating \eqref{eq:DB2} at $x=-D$ together with the boundary conditions \eqref{orisys2} and \eqref{tar2}, which yields  
\begin{align}\label{precontrol}
q_{c}(t) =& - c k \left(\int_{-D}^{0} v(y,t) dy + \frac{1}{\alpha} \int_{0}^{s(t)}  u(y,t) dy + \frac{1}{\beta} X(t) \right). 
\end{align}
 Finally, substituting the definitions \eqref{state} and \eqref{vdef} in \eqref{precontrol}, the control law \eqref{delaycomp} is obtained. 
 
 In a similar manner, the inverse transformations are obtained by 
\begin{align}\label{inv1}
u(x,t) =& w(x,t) + \frac{\beta}{\alpha} \int_x^{s(t)} \psi(x-y) w(y,t) dy + \psi(x-s(t))X(t), \\
\label{inv2}v(x,t) =& z(x,t) - \int_{x}^{0} \mu (x-y) z(y,t) dy -\frac{\beta}{\alpha} \mu(x) \int_{0}^{s(t)} \zeta(y)  w(y,t) dy - \zeta(s(t)) \mu(x) X(t), 
\end{align}
where 
\begin{align} \label{psidef} 
\psi(x) =&  \frac{\sqrt{c\alpha}}{\beta} {\rm sin}\left( \sqrt{\frac{c}{\alpha}}x\right), \\
\label{mudef} \mu (x) =& ce^{cx}, \quad \zeta(x) = \frac{1}{\beta} {\rm cos}\left( \sqrt{\frac{c}{\alpha}}x\right). 
\end{align}

\section{Physical constraints}\label{sec:constraints}
Noting that  \eqref{valid2} should hold since $q_{{\rm c}}(t)>0$ is required by Remark \ref{assumption} and Lemma \ref{lem1} to satisfy \eqref{valid1}, the overshoot beyond the setpoint $s_{{\rm r}}$ is prohibited to achieve the control objective $s(t) \to s_{{\rm r}}$, i.e.  $s(t)<s_{{\rm r}}$ is required to be satisfied for $\forall t>0$. In this section, we prove that the closed-loop system with the proposed control law \eqref{delaycomp} guarantees $q_{{\rm c}}(t)>0$ and $s(t)<s_{{\rm r}}$ for $\forall t>0$, namely "physical constraints".
\begin{lemma}\label{Prop1}
With Assumption \ref{assum1} and \ref{assum2}, the control law (14) for the system (1)-(5) generates a positive input signal, i.e.,  
\begin{align}\label{constraint1}
q_{{\rm c}}(t)>&0, \quad \forall t>0. 
\end{align}
\end{lemma}
\begin{proof}
Taking the time derivative of \eqref{delaycomp} together with the solution of \eqref{eq:stefanPDE}--\eqref{eq:stefanODE}, we obtain
\begin{align}\label{dercont}
\dot{q}_{{\rm c}}(t) = -c q_{{\rm c}}(t). 
\end{align}
The differential equation \eqref{dercont} yields $q_{{\rm c}}(t) = q_{{\rm c}}(0) e^{-ct}$. Additionally, Assumption \ref{assum2} leads to $q_{{\rm c}}(0)>0$. Thus, the positivity of the controller \eqref{constraint1} is satisfied. 
\end{proof}
\begin{lemma}\label{lem:temp}
The following property of liquid temperature profile holds: 
\begin{align}
\label{constraint2}T(x,t) \geq& T_{{\mathrm m}} \quad \textrm{ for all }\quad x \in [0,s(t)]. 
\end{align}
\end{lemma}
\begin{proof}
Applying Lemma \ref{Prop1} with Assumption \ref{assum1} to Lemma \ref{lem1}, \eqref{constraint2} is shown directly. 
\end{proof}
\begin{lemma}\label{lem:interface}
The following properties of the moving interface hold:
\begin{align}
\label{constraint3}\dot{s}(t) \geq& 0, \quad \forall t>0, \\
\label{constraint4}s_0<&s(t)<s_{{\rm r}}, \quad \forall t>0. 
\end{align}
\end{lemma}
\begin{proof}
Applying Lemma \ref{lem:temp} to Lemma \ref{monoinc}, \eqref{constraint3} is derived. Then, the condition \eqref{constraint3} leads to $s_0 < s(t)$. Finally, applying \eqref{constraint1} and \eqref{constraint2} to the control law \eqref{delaycomp}, \eqref{constraint4} is derived. 
\end{proof}

\section{Stability Analysis}\label{stability}
In this section, we derive the conclusion of Theorem \ref{thm1} by applying Lyapunov analysis and showing the norm equivalence with the help of \eqref{constraint3} and \eqref{constraint4}.  

\subsection{Change of variable}
Introduce a change of variable 
\begin{align}\label{omedef}
\omega(x,t) = w(x,t) + \left(x - s(t) \right) z(0,t). 
\end{align}
Using \eqref{omedef}, the target $(z,w,X)$-system \eqref{tar1}--\eqref{tarODE} is described by $(z,\omega,X)$-system as 
\begin{align}\label{chsys1}
z(-D,t) =& 0, \\
\label{chsys2}z_t(x,t) =& -z_x(x,t),\quad -D<x<0\\
\label{chsys3}\omega_x(0,t) =& 0,\\
\omega_{t}(x,t) =& \alpha \omega_{xx}(x,t) - \left(x - s(t) \right) z_{x}(0,t) + \dot{s}(t) \left( \frac{c}{\beta} X(t) - z(0,t)\right) ,  \quad 0<x<s(t)\label{chsys4}\\
\label{chsys5}\omega(s(t),t) =&0,\\
\label{chsys6}\dot{X}(t) =& -cX(t) -\beta (\omega_{x}(s(t),t)- z(0,t)). 
\end{align}

\subsection{Stability analysis of $(z, \omega, X)$-system}
Firstly, we prove the exponential stability of the $(z, \omega, X)$-system. Let $V_1$ be the functional defined by
\begin{align}\label{V1def}
V_1 = \int_{-D}^{0} e^{- mx }z_{x}(x,t)^2 dx, 
\end{align}
where $m>0$ is a positive parameter. \eqref{V1def} satisfies
\begin{align}
||z_{x}||_{{ L}_2 (-D,0)}^2 \leq V_1 \leq e^{mD} ||z_{x}||_{{ L}_2 (-D,0)}^2.
\end{align} 
Note that \eqref{chsys1} yields $z_{x}(-D,t) = 0$ through taking the time derivative and applying PDE \eqref{chsys2}. With the help of it, taking the time derivative of \eqref{V1def} together with \eqref{chsys1}-\eqref{chsys2} leads to 
\begin{align}\label{V1der}
\dot{V}_1 =& -2  \int_{-D}^{0} e^{- mx }z_{x}(x,t)z_{xx}(x,t) dx \notag\\
=& - e^{- mx }z_{x}(x,t)^2|_{x=-D}^{x=0} + \int_{-D}^{0} \left(\frac{d}{dx} e^{- mx } \right) z_{x}(x,t)^2 dx \notag\\
=& -  z_{x}(0,t)^2  - m \int_{-D}^{0} e^{-mx }z_{x}(x,t)^2 dx. 
\end{align}
Let $V_2$ be the functional defined by
\begin{align}\label{V2def}
V_2 =\frac{1}{2} \left( \frac{1}{s_r^2} ||\omega ||_{{ L}_2 (0,s(t))}^2 + ||\omega_{x} ||_{{ L}_2 (0,s(t))}^2 \right) =  \frac{1}{2}\int_0^{s(t)} \left( \frac{1}{s_r^2} \omega(x,t)^2 + \omega_{x}(x,t)^2 \right) dx.
\end{align}
\eqref{V2def} satisfies $\max\{s_r^2,1\}  ||\omega ||_{{\cal H}_1 (0,s(t))}^2 \leq 2 V_2 \leq \max\{1/s_r^2,1\}  ||\omega ||_{{\cal H}_1 (0,s(t))}^2 $   Note that taking the total time derivative of \eqref{chsys5} yields $\omega_{t}(s(t),t) = - \dot{s}(t) \omega_{x}(s(t),t)$. Taking the time derivative of \eqref{V2def} together with \eqref{chsys3}-\eqref{chsys5}, we obtain
\begin{align}\label{V2timeder} 
\dot{V}_2 = &\frac{\dot{s}(t)}{2} \left( \frac{1}{s_r^2} \omega(s(t),t)^2 + \omega_{x}(s(t),t)^2 \right) + \int_0^{s(t)} \left( \frac{1}{s_r^2} \omega(x,t)\omega_{t}(x,t) + \omega_{x}(x,t)\omega_{xt}(x,t) \right) dx \notag\\
= &\frac{\dot{s}(t)}{2}  \omega_{x}(s(t),t)^2  + \frac{1}{s_r^2} \int_0^{s(t)} \omega(x,t) \left( \alpha \omega_{xx}(x,t) - \left(x - s(t) \right) z_{x}(0,t) + \dot{s}(t) \left( \frac{c}{\beta} X(t) - z(0,t)\right) \right) dx \notag\\
&+  \omega_{x}(s(t),t)\omega_{t}(s(t),t) - \omega_{x}(0,t)\omega_{t}(0,t) -  \int_0^{s(t)} \omega_{xx}(x,t)\omega_{t}(x,t) dx \notag\\
=&- \frac{\alpha}{s_r^2} ||\omega_{x}||_{{L}_2 (0,s(t))}^2 - \frac{1}{s_r^2} z_{x}(0,t) \int_0^{s(t)} \left(x - s(t) \right) \omega(x,t) dx + \frac{\dot{s}(t)}{s_r^2} \left(\frac{c}{\beta} X(t) - z(0,t)\right) \int_0^{s(t)} \omega(x,t) dx \notag\\
& -\alpha ||\omega_{xx}||_{L_2 (0,s(t))}^2 +  z_{x}(0,t)\omega(0,t)-\frac{\dot{s}(t)}{2}\omega_x(s(t),t)^2 - \dot{s}(t) \left( \frac{c}{\beta} X(t) - z(0,t)\right)\omega_{x}(s(t),t). 
\end{align} 
Applying Young's and Cauchy Schwarz inequalities to the second terms on the first and second line of the  \eqref{V2timeder} with the help of \eqref{constraint4} yields  
\begin{align} \label{young11}
\bigg| z_{x}(0,t) \int_0^{s(t)} \left(x - s(t) \right) \omega(x,t) dx \bigg| &\leq \frac{\gam_1}{2} z_{x}(0,t)^2 + \frac{1}{2\gam_1} \left( \int_0^{s(t)} \left(x - s(t) \right) \omega(x,t) dx \right)^2, \notag\\
& \leq \frac{\gam_1}{2} z_{x}(0,t)^2 + \frac{1}{2\gam_1}  \left( \int_0^{s(t)} \left(x - s(t) \right)^2 dx \right) \left( \int_0^{s(t)} \omega(x,t)^2 dx \right), \notag\\
& \leq \frac{\gam_1}{2} z_{x}(0,t)^2 + \frac{s_r^3}{6\gam_1 }||\omega ||_{{ L}_2 (0,s(t))}^2, \notag\\
& \leq \frac{\gam_1}{2} z_{x}(0,t)^2 + \frac{2s_r^5}{3\gam_1 }||\omega_{x} ||_{{ L}_2 (0,s(t))}^2, 
\end{align} 
\begin{align} \label{young12} 
 |z_{x}(0,t)\omega(0,t)|& \leq \frac{\gam_2}{2} z_{x}(0,t)^2 + \frac{1}{2 \gam_2} \omega(0,t)^2, \notag\\
 &  \leq \frac{\gam_2}{2} z_{x}(0,t)^2 + \frac{2 s_r}{ \gam_2} ||\omega_{x}||_{{ L}_2 (0,s(t))}^2, 
 \end{align} 
where we utilized Poincare's inequality $||\omega ||_{{ L}_2 (0,s(t))}^2 \leq 4 s_r^2 ||\omega_{x} ||_{{ L}_2 (0,s(t))}^2$ and Agmon's inequality $ \omega(0,t)^2 \leq 4 s_r ||\omega_{x}||_{{ L}_2 (0,s(t))}^2$, and $\gam_1 >0$ and $\gamma_{2}>0$ are positive parameters to be determined. Hence, applying \eqref{young11} and \eqref{young12} to \eqref{V2timeder} with the choice of $\gam_1 = \frac{8 s_r^5}{3 \alpha} $ and $\gam_2 = \frac{8 s_r^3}{\alpha}$, the following differential inequality is deduced 
\ba 
\dot{V}_2 \leq& - \frac{\alpha}{2} ||\omega_{xx}||_{L_2 (0,s(t))}^2 - \frac{\alpha}{2 s_r^2} ||\omega_{x}||_{L_2 (0,s(t))}^2   + \frac{16 s_{{\rm r}}^3}{3 \alpha} z_{x}(0,t)^2  \notag\\
& +  \dot{s}(t) \left( 2 \frac{c^2}{\beta^2} X(t)^2 + 2 z(0,t)^2  + \frac{1}{2s_r^3}  ||\omega||_{L_2 (0,s(t))}^2\right) .\label{V2der}
\end{align}
Let $V_3$ be the functional defined by
\begin{align}\label{V3def}
V_3=\frac{1}{2}X(t)^2. 
\end{align}
Taking the time derivative of \eqref{V3def} and applying Young's and Agmon's inequalities, we obtain 
\begin{align}
\dot{V}_3=&-cX(t)^2-\beta X(t)  (\omega_{x}(s(t),t)- z(0,t)) \notag\\
\leq &-\frac{c}{2}X(t)^2+\frac{4\beta^2 s_{{\rm r}}}{c} ||\omega_{xx}||_{{ L}_2 (0,s(t))}^2+ \frac{4D\beta^2}{c} ||z_{x}||_{{ L}_2 (-D,0)}^2. \label{V3der}
\end{align}
Let $V$ be the functional defined by 
\begin{align}
V =qV_1 + V_2 + pV_3, 
\end{align}
where $q>0$ and $p>0$ are positive parameters to be determined. Combining \eqref{V1der}, \eqref{V2der}, and \eqref{V3der}, we get 
\begin{align}
\dot{V} \leq&  - \frac{\alpha}{2} \left(1  - \frac{8p\beta^2 s_{{\rm r}}}{c\alpha } \right) ||\omega_{xx}||_{L_2 (0,s(t))}^2 - \frac{\alpha}{2 s_r^2} ||\omega_{x}||_{L_2 (0,s(t))}^2   - \left( q -  \frac{16 s_{{\rm r}}^3}{3 \alpha} \right) z_{x}(0,t)^2  \notag\\
&  - m \left( q -  p \frac{4D\beta^2}{m c} \right) ||z_{x}||_{{ L}_2 (-D,0)}^2 -\frac{pc}{2}X(t)^2 \notag\\
& +  \dot{s}(t) \left( 2 \frac{c^2}{\beta^2} X(t)^2 + 2 z(0,t)^2  + \frac{1}{2s_r^3} \int_0^{s(t)} \omega(x,t)^2 dx\right) . \label{Vdotineq} 
  \end{align}
  Hence, by choosing the parameters as 
  \ba 
  p=  \frac{ c\alpha}{16 \beta^2 s_{{\rm r}}}, \quad  q = \max\left\{ \frac{16 s_{{\rm r}}^3}{3 \alpha}, \frac{D \alpha }{ 2m s_r}\right\} , 
  \end{align} 
  the inequality \eqref{Vdotineq} leads to 
  \begin{align}
\dot{V} \leq&  - \frac{\alpha}{4} ||\omega_{xx}||_{L_2 (0,s(t))}^2 - \frac{\alpha}{2 s_r^2} ||\omega_{x}||_{L_2 (0,s(t))}^2   - m \left( q -  p \frac{4D\beta}{m c} \right) ||z_{x}||_{{ L}_2 (-D,0)}^2 -\frac{pc}{2}X(t)^2 \notag\\
& +  \dot{s}(t) \left( 2 \frac{c^2}{\beta^2} X(t)^2 + 2 z(0,t)^2  + \frac{1}{2s_r^3} \int_0^{s(t)} \omega(x,t)^2 dx\right),  \notag\\
 \leq&  - \frac{\alpha}{ 8 s_r^2} V_2   - \frac{m q }{ 2} e^{-mD} V_1 -\frac{pc}{2}X(t)^2 +  \dot{s}(t) \left(  \frac{ 4c}{ \beta^2} V_3 + 8 D V_1  + \frac{1}{s_r} V_2\right) , 
  \end{align}
from which we obtain the form of
 \begin{align}
\dot{V} \leq & -bV + a\dot{s}(t) V, \label{Vder}
\end{align}
where 
\begin{align} 
b = \min \left\{\frac{m}{2} e^{-mD}, \frac{ \alpha}{8s_{{\rm r}}^2 } , c\right\}, \quad a = \max \left\{ \frac{8D}{q}, \frac{1}{s_{{\rm r}}}, \frac{ 4c^2}{p \beta^2} \right\}. 
\end{align} 
The differential inequality \eqref{Vder} does not directly lead to the exponential decay of the norm. To deal with it, we consider the following norm (of which the equivalence with $V$ follows from Lemma \ref{lem:interface})
\begin{align}\label{Wdef}
W = V e^{-a s(t)}. 
\end{align}
Taking the time derivative of \eqref{Wdef} with the help of \eqref{Vder}, we have 
\begin{align}\label{Wder}
\dot{W} = \left(\dot{V} - a\dot{s}(t) V \right) e^{-as(t)} \leq - bW. 
\end{align}
Hence, it leads to $W(t) \leq W(0) e^{-bt}$. Substituting \eqref{Wdef} and applying \eqref{constraint4}, the exponential stability of $(z, \omega, X)$-system is shown as
\begin{align}\label{Vexp}
V(t) \leq V(0) e^{as_{{\rm r}}} e^{-bt}. 
\end{align} 
\subsection{Stability analysis of $(z, w, X)$-system} \label{sec:stability2} 

Taking the square of \eqref{omedef} and applying Young's and Cauchy Schwarz inequality, we obtain
\begin{align}\label{omebound1}
|| \omega||_{{\cal H}_1 (0,s(t))}^2  \leq &2 || w||_{{\cal H}_1 (0,s(t))}^2 + K_{1} || z_{x}||_{{L}_2 (-D,0)}^2, \\
\label{omebound2}||w||_{{\cal H}_1 (0,s(t))}^2  \leq &2 || \omega ||_{{\cal H}_1 (0,s(t))}^2  + K_{1} || z_{x}||_{{ L}_2 (-D,0)}^2, 
\end{align}
where $K_{1} = \frac{8D s_{{\rm r}}^3}{3} +8D s_{{\rm r}}$. Consider the following norm 
\begin{align} \label{Pidef} 
\Pi (t)= ||z_x||_{{ L}_2 (-D,0)}^2 + ||w||_{{\cal H}_1 (0,s(t))}^2 + X(t)^2. 
\end{align}
Then, recalling $||z_{x}||_{{ L}_2 (-D,0)}^2 \leq V_1 \leq e^{mD} ||z_{x}||_{{ L}_2 (-D,0)}^2 $ and $K_2  ||\omega ||_{{\cal H}_1 (0,s(t))}^2 \leq 2 V_2 \leq K_3  ||\omega ||_{{\cal H}_1 (0,s(t))}^2$ where $ K_2 = \max\{s_r^2,1\}$ and $K_3 = \max\{1/s_r^2,1\}$, applying \eqref{omebound2} to \eqref{Pidef} yields the following bound: 
\begin{align} 
\Pi \leq &  (1+K_1)  ||z_{x}||_{{ L}_2 (-D,0)}^2  +2 ||\omega ||_{{\cal H}_1 (0,s(t))}^2 + X(t)^2, \notag\\
\leq & (1+K_1)  V_1  +4 K_2 V_2 + 2 V_3 . \label{Pibound} 
\end{align}
Moreover, recalling $ V = q V_1 + V_2 + p V_3$ and applying the above inequalities, the following bound on $V$ is derived:
\begin{align} 
V \leq &q e^{mD} || z_{x}||_{{ L}_2 (-D,0)}^2 + \frac{K_3}{2} || \omega||_{{\cal H}_1 (0,s(t))}^2 + \frac{p}{2} X(t)^2,  \notag\\
 \leq &\left(q e^{mD} + \frac{K_1 K_3 }{2} \right) || z_{x}||_{{ L}_2 (-D,0)}^2 + \frac{K_3}{2} || w||_{{\cal H}_1 (0,s(t))}^2   + \frac{p}{2} X(t)^2.  \label{Vbound} 
\end{align} 
Therefore, \eqref{Pibound} and \eqref{Vbound} leads to the following equivalence of the norm $V$ and $\Pi$: 
 \begin{align}\label{piineq}
 \underline{\delta} V(t) \leq \Pi(t) \leq \bar{\delta} V(t),
 \end{align}
where $ \underline{\delta} = \frac{1}{\max \left\{ q e^{mD} + \frac{K_1 K_3 }{2}, K_3, \frac{p}{2} \right\}}$ and $\bar{\delta} = \max \left\{ \frac{1}{q} \left(K_1 + 1\right), 4 K_2, \frac{2}{p} \right\}$. By \eqref{Vexp} and \eqref{piineq}, we have
 \begin{align}\label{Pidecay}
 \Pi (t) \leq \frac{\bar{\delta}}{\underline{\delta} } \Pi(0) e^{as_{{\rm r}}} e^{-bt}, 
 \end{align}
 which yields the exponential stability of $(z, w, X)$-system. 
\subsection{Stability analysis of $(v, u, X)$-system} \label{sec:stability3} 
Taking the spatial derivative of the transformation \eqref{eq:DB2} in $x$ leads to 
\begin{align} \label{zxvx}
z_{x}(x,t) = v_{x}(x,t) - c v(x,t).
\end{align} 
 Taking the square on both sides of \eqref{zxvx} and applying Young's inrquality, the following bound is obtained:
\begin{align}\label{zxbound}
||z_{x}||_{{ L}_2 (-D,0)}^2 \leq 2 ||v_{x}||_{{ L}_2 (-D,0)}^2 + 2 c^2  ||v||_{{ L}_2 (-D,0)}^2 . 
\end{align}
By the inverse transformation \eqref{inv2} and Poincare's inequality, we have 
 \begin{align}\label{vL2}
  ||v||_{{ L}_2 (-D,0)}^2 \leq& N_1 ||z_{x}||_{{ L}_2 (-D,0)}^2 +N_2 || w||_{{ L}_2 (0,s(t))}^2 + N_3 X(t)^2,
\end{align}
where $N_1 = 16D^2 \left(1 + \frac{cD}{2} (1-e^{-2D})\right)$, $N_2 = \frac{2 c s_{{\rm r}}}{ \alpha^2} (1-e^{-2D})$, and $N_3 =  \frac{2c}{ \beta^2} (1-e^{-2D})$. Also, by rewriting \eqref{zxvx} as $v_x(x,t) = z_x(x,t) + c v(x,t)$, the following inequality is derived:
\begin{align}\label{vH1before}
||v||_{{\cal H}_1 (-D,0)}^2  \leq  2 ||z_x||_{{ L}_2 (-D,0)}^2 + (2c^2+1) ||v||_{{ L}_2 (-D,0)}^2.
\end{align}
 Combining \eqref{vL2} with \eqref{vH1before}, the following inequality holds
 \begin{align}\label{vH1}
  ||v||_{{\cal H}_1 (-D,0)}^2 \leq & (2+(2c^2+1) N_1)  ||z_x||_{{ L}_2 (-D,0)}^2 +  (2c^2+1) (N_2 || w||_{{ L}_2 (0,s(t))}^2  + N_3 X(t)^2). 
 \end{align}
 Moreover, as shown in \cite{Shumon16}, there exist positive constants $M_{i}>0$ for $i = 1, 2, 3, 4$ such that 
\begin{align}\label{bound1}
||w||_{{\cal H}_1 (0,s(t))}^2 &\leq  M_1 ||u||_{{\cal H}_1 (0,s(t))}^2 +M_2 X(t)^2,\\
\label{bound3}||u||_{{\cal H}_1 (0,s(t))}^2  &\leq  M_3 ||w||_{{\cal H}_1 (0,s(t))}^2 + M_4 X(t)^2. 
\end{align}
Then, adding \eqref{zxbound} to \eqref{bound1} and \eqref{vH1} to \eqref{bound3}, we have 
\begin{align}\label{zwnorm}
 ||z_{x}||_{{ L}_2 (-D,0)}^2  + || w|| _{{\cal H}_1 (0,s(t))}^2 + X(t)^2 \leq& 2 ||v_x||_{{ L}_2 (-D,0)}^2 + 2c^2 ||v||_{{ L}_2 (-D,0)}^2 + M_1 ||u||_{{\cal H}_1 (0,s(t))}^2 + (M_2 + 1) X(t)^2, \\
\label{vunorm}
||v||_{{\cal H}_1 (-D,0)}^2  + ||u||_{{\cal H}_1(0,s(t))}^2  + X(t)^2  \leq &  L_1  ||z_x||_{{ L}_2 (-D,0)}^2 + L_2 || w||_{{\cal H}_1 (0,s(t))}^2   + L_3X(t)^2,
 \end{align}
 where $L_1 = 2+(2c^2+1) N_1$, $L_2 = (2c^2+1)N_2+M_3$, and $L_3 = (2c^2+1) N_3+M_4+1$. 
 Define the following norm
 \begin{align}
 \Xi(t) =&  ||v||_{{\cal H}_1 (-D,0)}^2  + || u||_{{\cal H}_1 (0,s(t))}^2 + X(t)^2. 
 \end{align}
Then, \eqref{zwnorm} and \eqref{vunorm} leads to 
\begin{align}\label{XiPi}
\underline{M} \Xi(t) \leq \Pi(t) \leq \bar M \Xi(t), 
\end{align}
where $\bar M = \max \left\{2, 2c^2, M_1, M_2 + 1\right\}$, $\underline{M} =\frac{1}{ \max \left\{ L_1, L_2, L_3 \right\} }$. Finally, applying \eqref{XiPi} to \eqref{Pidecay} we arrive at 
\begin{align}
\Xi(t) \leq \frac{\bar M}{\underline{M}} \frac{\bar \delta}{\underline{\delta}} \Xi(0) e^{a s_{{\rm r}}} e^{-bt}, \label{finallyap} 
\end{align}
which completes the proof of Theorem \ref{thm1}. 

\section{Relation between the designed control law and a state prediction} \label{sec:prediction} 
As developed in some literature for ODE systems, the delay compensated control via the method of backstepping is known to be equivalent to the predictor-based feedback where the control law is derived to stabilize the future state called "predictor state", see Section 2 in \cite{krstic2009delay} for instance. Hence, one might have a question whether our delay compensated control is also equivalent to the predictor-based feedback.  This is not a trivial question in the case of Stefan problem due to the complicated structure of ODE dynamics whose state is the domain of the PDE.

The nominal control design for delay-free Stefan problem developed in \cite{Shumon16} is given by 
\begin{align}\label{uncomp}
\bar q_{{\rm c}}(t) =& - c  \left( \frac{k}{\alpha} \int_{0}^{s(t)}  (T(x,t) - T_{{\rm m}}) dx + \frac{k}{\beta} (s(t) -s_{{\rm r}}) \right),  
\end{align}
where we defined the notation $\bar q_{{\rm c}}(t)$ to distinguish with the delay compensated control law \eqref{delaycomp}. Thus, our interest lies in proving $q_{{\rm c}}(t) \equiv \bar q_{{\rm c}}(t+D)$ because $ \bar q_{{\rm c}}(t+D)$ is the prediction of the nominal control. We start from the expression of $\bar q_{{\rm c}}(t+D)$ which can be described as 
\begin{align}
	\bar q_{{\rm c}}(t+D) =& - c  \left( \frac{k}{\alpha} \int_{0}^{s(t+D)}  (T(x,t+D) - T_{{\rm m}}) dx + \frac{k}{\beta} (s(t+D) -s_{{\rm r}}) \right). \label{predict0}
\end{align}
Integrating ODE dynamics $\dot{s}(t) = - \beta T_{x}(s(t),t)$ given in \eqref{eq:stefanODE} from $t$ to $t+D$ yields 
\begin{align}
s(t+D) = s(t) - \beta \int_t^{t+D} T_{x}(s(\tau), \tau) d \tau . 	\label{predict1} 
\end{align}
Next, integrating PDE dynamics $T_{t} = \alpha T_{xx}$ given in \eqref{eq:stefanPDE} in time from $t$ to $t+D$ leads to $T(x,t+D) = T(x,t) + \alpha \int_t^{t+D} T_{xx}(x,\tau) d\tau$. Furthermore, integrating the both sides in space from $0$ to $s(t+D)$, we obtain
\begin{align}
\int_{0}^{s(t+D)}  (T(x,t+D) - T_{{\rm m}}) dx =& \int_{0}^{s(t+D)}  (T(x,t) - T_{{\rm m}}) dx + \alpha \int_{0}^{s(t+D)}  \int_t^{t+D} T_{xx}(x,\tau) d\tau dx, \notag\\
= & \int_{0}^{s(t+D)}  (T(x,t) - T_{{\rm m}}) dx + \alpha \int_t^{t+D} (T_{x}(s(t+D),\tau) - T_{x}(0,\tau) d\tau, \notag\\
=& \int_{0}^{s(t+D)}  (T(x,t) - T_{{\rm m}}) dx + \alpha \int_t^{t+D} T_{x}(s(t+D),\tau) d\tau + \frac{\alpha}{k} \int_{t-D}^{t} q_{c}(\xi) d \xi .\label{predict2} 
\end{align}
Therefore, substituting \eqref{predict1} and \eqref{predict2} into \eqref{predict0}, we get 
\begin{align}
	\bar q_{{\rm c}}(t+D) =& - c  \left( \frac{k}{\alpha} \int_{0}^{s(t+D)}  (T(x,t) - T_{{\rm m}}) dx + k \int_t^{t+D} (T_{x}(s(t+D),\tau) - T_{x}(s(\tau), \tau) )d \tau + \int_{t-D}^{t} q_{c}(\xi) d \xi + \frac{k}{\beta} (s(t) -s_{{\rm r}}) \right). \label{predict3} 
\end{align}
Consequently, it remains to consider the following term 
\begin{align}
\int_t^{t+D} (T_{x}(s(t+D),\tau) - T_{x}(s(\tau), \tau) )d \tau =& \int_t^{t+D} \int_{s(\tau)}^{s(t+D)} T_{xx}(x, \tau)  dx d \tau, \notag\\
= & \frac{1}{\alpha} \int_{s(t)}^{s(t+D)} \int_t^{s^{-1}(x)}  T_{\tau}(x, \tau) d \tau dx, \notag\\
= & \frac{1}{\alpha} \int_{s(t)}^{s(t+D)} \left(T(x,s^{-1}(x)) - T(x,t)\right) dx.  \label{inverses} 
\end{align}
where we switched the order of the integrations in time and space from the first line to the second line with defining the inverse function $s^{-1}(x)$. The existence and uniqueness of $s^{-1}(x)$ is guaranteed due to the continuous and monotonically increasing property of $s(t)$ provided $q_{c}(t)>0$. Thus, boundary condition $T(s(t),t) = T_{m}$, $\forall t\geq 0$ given in \eqref{eq:stefanBC} implies $T(x,s^{-1}(x)) = T_{m}$ from which \eqref{inverses} is given by 
\begin{align}
	\int_t^{t+D} (T_{x}(s(t+D),\tau) - T_{x}(s(\tau), \tau) )d \tau =& - \frac{1}{\alpha} \int_{s(t)}^{s(t+D)} \left(T(x,t) - T_m\right) dx. \label{predict_fin} 
\end{align}
Substituting \eqref{predict_fin} into \eqref{predict3}, we arrive at
\begin{align}
	\bar q_{{\rm c}}(t+D) =& - c  \left( \frac{k}{\alpha} \int_{0}^{s(t)}  (T(x,t) - T_{{\rm m}}) dx  + \int_{t-D}^{t} q_{c}(\xi) d \xi + \frac{k}{\beta} (s(t) -s_{{\rm r}}) \right) \equiv q_{{\rm c}}(t) . 
\end{align}
Therefore, we conclude that the delay compensated control \eqref{delaycomp} is indeed the prediction of the nominal control law \eqref{uncomp}.

\section{Robustness to delay mismatch} \label{sec:robust} 
The results established up to the last section are based on the control design with utilizing the exact value of the actuator delay. However, in practice, there is an error between the exact time delays and the identified delays. Hence, guaranteeing the performance of the controller under the small delay mismatch is important. In this section, $D>0$ is denoted as the identified time delay and $\del D$ is denoted as the delay mismatch (can be either positive or negative), which yields $D + \Delta D$ as the exact time delay from the controller to the plant. Thus, the system we focus on is described by 
\begin{align}\label{robustsys1}
T_{t}(x,t) =& \alpha T_{xx}(x,t) , \quad x \in (0, s(t)) , \\
\label{robustsys2}-k T_{x}(0,t) =& q_{c}(t - (D + \Delta D)) , \\
\label{robustsys3}T(s(t),t) =& T_m , \\
\label{robustsys4}\dot{s}(t) =& - \beta T_{x}(s(t),t) , 
\end{align} 
with the control law given in \eqref{delaycomp} which utilizes the identified delay $D$. Since the control law is not changed, the same backstepping transformation in \eqref{eq:DB1} and \eqref{eq:DB2} can be applied, but the target $(z,w,X)$-system needs to be redescribed due to the modification of \eqref{robustsys2}. The theorem for the robustness to delay mismatch is provided under the restriction on the control gain, as stated in the following. 
\begin{theorem} \label{thm:robust} 
Under Assumptions \ref{initial}-\ref{assum2}, there exists a positive constant $\bar c>0$ such that $\forall c\in(0, \bar c)$ the closed-loop system consisting of the plant \eqref{robustsys1}--\eqref{robustsys4} and the control law \eqref{delaycomp} maintains the model validity \eqref{valid1} and is exponentially stable in the sense of the norm \eqref{h1}. 
\end{theorem}

An important characteristic to note in Theorem \ref{thm:robust} is that the existence of $\bar c$ is ensured for any given $\del D$ as long as $D + \del D>0$. An analogous description with respect to the small delay mismatch is given in the following corollary. 
\begin{corollary} 
Under Assumptions \ref{initial}-\ref{assum2}, for any given $c>0$ there exist positive constants $\underline{\eps} >0$ and $\bar{\eps}>0$ such that $\forall \del D \in (- \underline{\eps} , \bar{\eps} )$ the closed-loop system \eqref{robustsys1}--\eqref{robustsys4}, \eqref{delaycomp} satisfies the same model validity and stability property as Theorem \ref{thm:robust}. 
\end{corollary} 

The proof of Theorem \ref{thm:robust} is established through the remaining of this section. 

\subsection{Reference error system} 
Introduce the same definition of the reference errors as in \eqref{state} and \eqref{vdef}, namely, $u(x,t) = T(x,t) - T_m$, $X(t) = s(t) - s_r$. For the delayed input state, we modify to the following definition
\begin{align} \label{delaystate-rob}
v(x,t) = q_{c}(t - x - (D + \Delta D))/k.  
\end{align} 
Then, the system \eqref{robustsys1}--\eqref{robustsys4} is rewritten as a reference error $(u, v, X)$-system as  
\begin{align}\label{robusterrsys1}
v_t(x,t) =& -v_x(x,t),\quad - (D + \Delta D)<x< \max \{0, - \del D\}\\
\label{robusterrsys2}v(- (D + \Delta D),t) =& q_{{\rm c}}(t)/k, \\
\label{robusterrsys3} u_x(0,t) =& -v( 0,t),\\
\label{robusterrsys4}u_{t}(x,t) =& \alpha u_{xx}(x,t),  \quad 0<x<s(t)\\
\label{robusterrsys5}u(s(t),t) =&0,\\
\label{robusterrsys6}\dot{X}(t) =& -\beta u_{x}(s(t),t). 
\end{align}

\subsection{Target system} 
We apply the same transformations as in \eqref{eq:DB1}--\eqref{eq:DB2}. Since the boundary condition at the controller's position is replaced by \eqref{robusterrsys2}, evaluating \eqref{eq:DB2} at $x=- (D + \Delta D)$ yields 
\begin{align}\label{z-D}
z(- (D + \Delta D) ,t) =& \frac{q_{{\rm c}}(t)}{k} + c \int_{- (D + \Delta D)}^{0} v(y,t) dy+ \frac{c}{\alpha} \int_{0}^{s(t)}  u(y,t) dy + \frac{c}{\beta} X(t). 
\end{align}
Rewriting the control law \eqref{delaycomp} using the delayed input state \eqref{delaystate-rob}, we obtain 
\begin{align} \label{qck-rob}
\frac{q_{c}(t)}{k} = - c \left( \int_{- (D + \Delta D)}^{- \del D} v(y,t) dy+ \frac{1}{\alpha} \int_{0}^{s(t)}  u(y,t) dy + \frac{1}{\beta} X(t) \right) . 
\end{align} 
By \eqref{z-D} and \eqref{qck-rob}, it holds  
\begin{align} 
z(- (D + \Delta D),t) =& c \int_{- \Delta D}^{ 0} v(x,t) dx . 
\end{align} 
Let $f(t)$ be defined by 
\ba \label{ftdef} 
f(t):=  \int_{- \Delta D}^{ 0} v(x,t) dx. 
\end{align} 
To describe the closed-form of the target system, the formulation of \eqref{ftdef} needs to be rewritten with respect to the variables $(z,w,X)$. Applying the inverse transformation \eqref{inv2} to \eqref{ftdef} and calculating the integrations, we deduce (see Appendix \ref{app:f} for the derivation)
\begin{align} \label{ftfin} 
f(t) = -  \int_{- \Delta D}^{ 0} e^{- c(x + \Delta D)} z(x,t) dx - ( 1 - e^{ - c \Delta D} ) \left( \frac{\beta}{\alpha} \int_{0}^{s(t)} \zeta(y)  w(y,t) dy + \zeta(s(t)) X(t) \right)  . 
\end{align} 
 Then, we obtain the target $(z,w,X)$-system as 
\begin{align}\label{robusttar1}
z_t(x,t) =& -z_x(x,t),\quad - (D + \Delta D)<x< \max \{0, - \del D\}\\
\label{robusttar2}z(- (D + \Delta D),t) =& cf(t), \\
\label{robusttar3}w_x(0,t) =& -z(0,t),\\
\label{robusttar4}w_{t}(x,t) =& \alpha w_{xx}(x,t) + \frac{c}{\beta} \dot{s}(t) X(t),  \quad 0<x<s(t)\\
w(s(t),t) =&0,\\
\label{robusttarODE}\dot{X}(t) =& -cX(t) -\beta w_{x}(s(t),t). 
\end{align}

\subsection{Physical constraints} 
The conditions for the model validity \eqref{valid1} need to be satisfied by proving the positivity of the control law as explained in Section \ref{sec:constraints}. Taking the time derivative of the controller \eqref{delaycomp} along the solution of \eqref{robustsys1}--\eqref{robustsys4}, we obtain 
\begin{align} \label{robustcontder}
\dot{q}_{c}(t) = - c q_{c}(t) + c \left( q_{c}(t - D) - q_{c}(t - (D + \Delta D)) \right) . 
\end{align} 
The solution to \eqref{robustcontder} is hard to solve explicitly, however, it is possible to investigate the positivity of $q_{c}(t)$ {\em a priori} since the differential equation \eqref{robustcontder} has the closed form in $q_{c}$. To ensure it, the following lemma is deuced. 
\begin{lemma} \label{lem:robustpositive} 
For a given $\del D \in R$, there exists $c^*>0$ such that $\forall c \in (0, c^*)$ the solution to the delay differential equation \eqref{robustcontder} satisfies the positivity, i.e., $q_{c}(t) > 0$ for all $t \geq 0$. 
 \end{lemma}
Owing to Lemma \ref{lem:robustpositive}, with sufficiently small control gain $c>0$, the properties $\dot s(t)>0$ and $s_0<s(t)<s_r$ are satisfied $\forall t>0$ as explained in Lemma \ref{lem:interface}, which are utilized for the stability analysis. The proof of Lemma \ref{lem:robustpositive} is given in Appendix \ref{app:lemma}.

\subsection{Stability analysis} 

Introduce the same change of variable as \eqref{omedef}, i.e., $\omega(x,t) = w(x,t) + \left(x - s(t) \right) z(0,t)$. Then, the resulting target system becomes 
\begin{align}\label{robustchsys1}
z(- (D + \Delta D),t) =& c f(t), \\
\label{robustchsys2}z_t(x,t) =& -z_x(x,t),\quad - (D + \Delta D)<x< \max \{0, - \del D\}\\
\label{robustchsys3}\omega_x(0,t) =& 0 ,\\
\omega_{t}(x,t) =& \alpha \omega_{xx}(x,t) - \left(x - s(t) \right) z_{x}(0,t) + \dot{s}(t) \left( \frac{c}{\beta} X(t) - z(0,t)\right) ,  \quad 0<x<s(t)\label{robustchsys4}\\
\label{robustchsys5}\omega(s(t),t) =&0,\\
\label{robustchsys6}\dot{X}(t) =& -cX(t) -\beta (\omega_{x}(s(t),t)+ z(0,t)). 
\end{align}
Inequalities of $f(t)^2$ and $f'(t)^2$ are derived in Appendix \ref{app:bound}, and they are utilized in stability analysis. The form of the inequalities are slightly changed depending on $\del D>0$ or $\del D<0$, due to the domain $- (D + \Delta D)<x< \max \{0, - \del D\}$ for $z$-subsystem. Thus, the $(z, \omega, X)$-system's stability needs to be analyzed by separating the cases of $\del D>0$ (underestimated delay mismatch) and $\del D<0$ (over-estimated delay mismatch), as presented in \cite{krstic2008lyapunov} to prove delay-robustness for linear ODE systems. 

\subsubsection{Underestimated delay mismatch: $\del D>0$} 
Consider 
\begin{align}\label{robV1def}
V_1 = \int_{- (D+\Delta D)}^{0} e^{- mx }z(x,t)^2 dx. 
\end{align}
Taking the time derivative of \eqref{robV1def}, we have 
\begin{align} \label{robV1dot}
\dot{V}_1  = - z(0,t)^2 + e^{ m(D + \del D) } c^2 f(t)^2 - m V_1 . 
\end{align} 
Applying the bound given in Appendix \ref{app:fbound} with the help of $\del D>0$ and $||z||_{{ L}_2 (- \del D,0)}^2 \leq ||z||_{{ L}_2 (-(D+ \del D),0)}^2$ to \eqref{robV1dot} yields 
 \begin{align} \label{underestimateV1dot} 
\dot{V}_1  \leq  - z(0,t)^2 + e^{ m(D + \del D) } c^2 \left( 2 \bar M_1  ||z||_{{ L}_2 (-(D+ \del D),0)}^2 + \bar M_2 || \omega||^2 +  \bar M_3 z(0,t)^2  + \bar M_4 X(t)^2 \right) - m V_1 , 
\end{align} 
where 
\begin{align} \label{def:M} 
\bar M_1 = \frac{{\textrm{sign}} ( \del D)}{2c} \left( 1 - e^{-2c \del D} \right), \quad \bar M_2 = \frac{8 s_r}{\alpha^2}  ( 1 - e^{ - c \Delta D} )^2  , \quad \bar M_3 = \frac{ 8}{\alpha c}  ( 1 - e^{ - c \Delta D} )^2 , \quad \bar M_4 =  \frac{4}{\beta^2} ( 1 - e^{ - c \Delta D} )^2. 
\end{align} 
Next, we consider
\begin{align}\label{robV2def}
V_2 = \int_{- (D+\Delta D)}^{0} e^{- m x }z_{x}(x,t)^2 dx. 
\end{align}
\eqref{robV2def} satisfies $||z_{x}||_{{ L}_2 (-(D+\Delta D),0)}^2 \leq V_2 \leq e^{m(D+ \del D)} ||z_{x}||_{{ L}_2 (-(D+ \del D),0)}^2$. Taking the time derivative of \eqref{robV2def} together with \eqref{robustchsys1}-\eqref{robustchsys2}, we have 
\begin{align}\label{robV2der}
\dot{V}_2 = & -  z_{x}(0,t)^2  +  e^{ m(D + \del D) } c^2 f'(t)^2 - m V_2. 
\end{align}
Applying the bound given in Appendix \ref{app:ftbound} to \eqref{robV1dot} yields 
\begin{align} 
\dot{V}_2 \leq &-  z_{x}(0,t)^2  - m V_2  \notag\\
& +  e^{ m(D + \del D) } c^2 \left(  4 | \del D| ||z_{x}||_{{ L}_2 (-(D+ \del D),0)}^2 + 2 c^2 \bar M_1 ||z||_{{ L}_2 (-(D+ \del D),0)}^2  +  c^2 \left(  \bar M_2 || \omega||^2 +  \bar M_3 z(0,t)^2  + \bar M_4 X(t)^2\right)  \right). \label{underestimateV2dot} 
\end{align} 
Hence, we have 
\begin{align} 
\dot{V}_1 + r \dot{V}_2 \leq&  -  \left( 1 - c^2e^{ m(D + \del D)} \bar M_3 (1 + r c^2)  \right) z(0,t)^2 + e^{ m(D + \del D) } c^2 \bar M_2(1+rc^2)  || \omega||^2 \notag\\
& - ( m - 2e^{ m(D + \del D) } c^2 \bar M_1 (1+rc^2)) V_1 + e^{ m(D + \del D) } c^2 \bar M_4(1+rc^2)   X(t)^2   \notag\\
&- r  z_{x}(0,t)^2  - r  (m -  4 e^{ m(D + \del D) } c^2 \del D)  V_2  \label{robunderV1dot} 
\end{align} 
Since $(\omega,X)$-subsystem \eqref{robustchsys3}--\eqref{robustchsys6} is the same formulation as in the exact prediction case, using the same norm $V_3 =\frac{1}{2} ||\omega ||_{{\cal H}_1 (0,s(t))}^2 =  \frac{1}{2}\int_0^{s(t)} ( \frac{\omega(x,t)^2}{s_r^2}+\omega_{x}(x,t)^2) dx$ and $ V_4=\frac{1}{2}X(t)^2$, 
the time derivative of them yields the same norm estimate as in the exact prediction case, which are 
\begin{align}
\dot{V}_3 \leq& - \frac{\alpha}{2} ||\omega_{xx}||_{L_2 (0,s(t))}^2 - \frac{\alpha}{2 s_r^2} ||\omega_{x}||_{L_2 (0,s(t))}^2   + \frac{16 s_{{\rm r}}^3}{3 \alpha} z_{x}(0,t)^2  \notag\\
& +  \dot{s}(t) \left( 2 \frac{c^2}{\beta^2} X(t)^2 + 2 z(0,t)^2  + \frac{1}{2s_r^3}  ||\omega||_{L_2 (0,s(t))}^2\right) \label{d-V2der}, \\
\dot{V}_4\leq &-\frac{c}{2}X(t)^2+\frac{4\beta^2 s_{{\rm r}}}{c} ||\omega_{xx}||_{{ L}_2 (0,s(t))}^2+ \frac{\beta^2}{c} z(0,t)^2. \label{robunderV4dot}
\end{align}
Let $V$ be the Lyapunov function defined by 
\begin{align}\label{robunderV} 
V =V_1 + r V_2 + qV_3 + pV_4,  
\end{align}
where $r>0$, $q>0$, $p>0$ are positive parameters to be determined. Applying \eqref{robunderV1dot}--\eqref{robunderV4dot} to the time derivative of \eqref{robunderV} with choosing 
\begin{align} \label{robust:parameters} 
 q = \frac{8 s_r}{\alpha}, \quad p  =  \frac{c }{2 \beta^2 }, \quad r = \frac{32q s_{{\rm r}}^3}{3 \alpha} , 
\end{align} 
we get  
\begin{align} 
\dot{V} \leq  &  -  \left( \frac{1}{2} -  \frac{ 8}{\alpha } e^{ m(D + \del D)} c ( 1 - e^{ - c \Delta D} )^2  (1 + r c^2)  \right) z(0,t)^2 -  \frac{r}{2} z_{x}(0,t)^2 \notag\\
& - ( m - e^{ m(D + \del D) }  c\left( 1 - e^{-2c \del D} \right) (1+rc^2)) V_1  \notag\\
& - r \left(m -  4 e^{ m(D + \del D) } c^2 |\del D|  \right) V_2  - \frac{c^2}{4 \beta^2 } \left\{ 1  - 16 e^{ m(D + \del D) } ( 1 - e^{ - c \Delta D} )^2 (1+rc^2)  \right\}X(t)^2  \notag\\
& -  2 s_r || \omega_{xx}||_{L_2 (0,s(t))}^2  - 2 s_r || \omega_{x}||_{L_2 (0,s(t))}^2 - \left( \frac{1}{2 s_r} -  \frac{8 s_r}{\alpha^2} e^{ m(D + \del D) } c^2  ( 1 - e^{ - c \Delta D} )^2(1+rc^2)\right)  || \omega||^2   \notag\\
&+q\dot{s}(t) \left\{  \frac{2 c^2}{\beta^2} X(t)^2  + 8D|| z_{x}||_{{ L}_2 (-D,0)}^2+ \frac{s_{{\rm r}}}{2} || \omega||_{{ L}_2 (0,s(t))}^2  \right\} , \label{robust_under_Vdot} 
\end{align} 
where we substitute the definition of $\bar M_{i}$ for $i=1,2,3,4$ given in \eqref{def:M}. Thus, for sufficiently small $c>0$, there exist positive constants $b>0$ and $a>0$ such that $\dot{V} \leq - bV + a\dot{s}(t) V$ holds. Using the same technique as deriving from \eqref{Vder} to \eqref{finallyap} in exact compensation case, we conclude Theorem \ref{thm:robust} for underestimated delay mismatch. 

\subsubsection{Over-estimated delay mismatch: $\del D<0$} 
We consider the same definitions of $V_1$, $V_2$, $V_3$, $V_4$ as those in the case $\del D>0$. However, in the case of $\del D<0$, the bounds given in Appendix \ref{app:bound} yields $f(t)^2 \leq  2 \bar M_1   ||z||_{{ L}_2 (0, - \Delta D)}^2 + \bar  M_2 || \omega||^2 + \bar M_3 z(0,t)^2  + \bar  M_4 X(t)^2$ and $ f'(t)^2 \leq  4 | \del D| ||z_{x}||_{{ L}_2 (0, - \del D)}^2 + 2 c^2 \bar M_1 ||z||_{{ L}_2 (0, \del D)}^2  +  c^2 \left(  \bar M_2 || \omega||^2 + \bar  M_3 z(0,t)^2  + \bar  M_4 X(t)^2\right) $. Hence, the bounds of \eqref{underestimateV1dot} and \eqref{underestimateV2dot} are replaced by 
\begin{align}
\dot{V}_1  \leq &  - z(0,t)^2 + e^{ m(D + \del D) } c^2 \left( 2 \bar  M_1  ||z||_{{ L}_2 (0, - \Delta D)}^2 + \bar M_2 || \omega||^2 + \bar  M_3 z(0,t)^2  +\bar M_4 X(t)^2 \right) - m V_1 , \\
\dot{V}_2 \leq &-  z_{x}(0,t)^2  - m V_2  \notag\\
& +  e^{ m(D + \del D) } c^2 \left(  4 | \del D| ||z_{x}||_{{ L}_2 (0, - \del D)}^2 + 2 c^2 \bar M_1 ||z||_{{ L}_2 (0, \del D)}^2  +  c^2 \left( \bar  M_2 || \omega||^2 + \bar M_3 z(0,t)^2  + \bar M_4 X(t)^2\right)  \right). 
\end{align} 
We additionally consider
\begin{align} \label{V5} 
V_5 = \int_{0}^{- \Delta D} e^{- m x} z(x,t)^2 dx , \\
V_6 =  \int_{0}^{- \Delta D} e^{- m x} z_{x}(x,t)^2 dx. \label{V6} 
\end{align} 
Then, we have $ e^{m \del D} ||z||_{{ L}_2 (0, - \Delta D)}^2 \leq V_5 \leq   ||z||_{{ L}_2 (0, - \Delta D)}^2$ and $ e^{m \del D} ||z_{x}||_{{ L}_2 (0, - \Delta D)}^2 \leq V_6 \leq   ||z_{x}||_{{ L}_2 (0, - \Delta D)}^2$. The time derivatives of \eqref{V5} and \eqref{V6} are given by 
\begin{align} 
\dot{V}_5 =& - e^{ m \Delta D} z( -\del D,t)^2 +   z( 0,t)^2 - m V_5, \\
\dot{V}_6 =& - e^{ m \Delta D} z_{x}( -\del D,t)^2 +    z_{x}( 0,t)^2 - m V_6. 
\end{align} 
Then, by redefining the Lyapunov function $V$ as 
\begin{align} \label{robust_V_2} 
V = V_1 + r V_2 + qV_3 + pV_4 + \frac{1}{4} V_5 +  \frac{r}{4}  V_6, 
\end{align}
with the same choices of the parameters as \eqref{robust:parameters}, the time derivative of \eqref{robust_V_2} satisfies the following inequality 
\begin{align} 
\dot{V} \leq  &  -  \left( \frac{1}{4} -  \frac{ 8}{\alpha } e^{ m(D + \del D)} c ( 1 - e^{ - c \Delta D} )^2  (1 + r c^2)  \right) z(0,t)^2 -  \frac{r}{4} z_{x}(0,t)^2 \notag\\
& - m V_1 - r m V_2  - \frac{c^2}{4 \beta^2 } \left\{ 1  - 16 e^{ m(D + \del D) } ( 1 - e^{ - c \Delta D} )^2 (1+rc^2)  \right\}X(t)^2  \notag\\
& - \left( \frac{m}{4} -  e^{ mD } c\left( e^{-2c \del D} - 1\right) (1+rc^2) \right) V_5 - \left( \frac{m}{4} -  4 r e^{ mD } c^2 |\del D| \right) V_6 \notag\\ 
& -  2 s_r || \omega_{xx}||_{L_2 (0,s(t))}^2  - 2 s_r || \omega_{x}||_{L_2 (0,s(t))}^2 - \left( \frac{1}{2 s_r} -  \frac{8 s_r}{\alpha^2} e^{ m(D + \del D) } c^2  ( 1 - e^{ - c \Delta D} )^2(1+rc^2)\right)  || \omega||^2   \notag\\
&+q\dot{s}(t) \left\{  \frac{2 c^2}{\beta^2} X(t)^2  + 8D|| z_{x}||_{{ L}_2 (-D,0)}^2+ \frac{s_{{\rm r}}}{2} || \omega||_{{ L}_2 (0,s(t))}^2  \right\} , 
\end{align} 
from which we conclude Theorem \ref{thm:robust} for over-estimated delay mismatch. 

\section{Numerical Simulation}\label{simulation}
In this section, we provide the simulation results of the proposed delay compensated controller under the accurate value on the delay and the delay mismatch. 
\subsection{Exact Compensation} 
The performance of the proposed delay compensated controller is investigated by comparing to the performance of the nominal controller \eqref{uncomp}. As in \cite{maidi2014} and \cite{Shumon16}, the simulation is performed considering a strip of zinc whose physical properties are given in Table \ref{zinc} using the boundary immobilization method with a finite difference discretization studied in \cite{kutluay97}. The time delay, the past heat input, and the initial values are set as $D =$ 2 [min], $q_{{\rm c}}(t) =$ 500 [W/m] for $\forall t\in[-D,0)$, $s_0$ = 0.1 [m], and $T_0(x)= \bar T(1-x/s_0) + T_{{\rm m}}$ with $\bar T =$ 50 [K]. The setpoint and the controller gain are chosen as $s_{{\rm r}} =$ 0.15 m and $c=$0.01/s, which satisfies the setpoint restriction \eqref{setpoint}. 

\begin{table}[t]
\vspace{2.5mm}
\caption{Physical properties of zinc}
\begin{center}
    \begin{tabular}{| l | l | l | }
    \hline
    $\textbf{Description}$ & $\textbf{Symbol}$ & $\textbf{Value}$ \\ \hline
    Density & $\rho$ & 6570 ${\rm kg}\cdot {\rm m}^{-3}$\\ 
    Latent heat of fusion & $\Delta H^*$ & 111,961 ${\rm J}\cdot {\rm kg}^{-1}$ \\ 
    Heat Capacity & $C_p$ & 389.5687 ${\rm J} \cdot {\rm kg}^{-1}\cdot {\rm K}^{-1}$  \\  
    Thermal conductivity & $k$ & 116 ${\rm w}\cdot {\rm m}^{-1}$  \\ \hline
    \end{tabular}\label{zinc}
\end{center}
\end{table}

Fig. \ref{fig:delay} shows the simulation results of the closed-loop system of the plant \eqref{eq:stefanPDE}--\eqref{eq:stefanODE} with the proposed delay compensated control \eqref{delaycomp} (red) and the uncompensated control law \eqref{uncomp} (blue). The closed-loop responses of the moving interface $s(t)$, the boundary heat control $q_{{\rm c}}(t)$, and the boundary temperature $T(0,t)$ are depicted in Fig. \ref{fig:delay} (a)--(c), respectively. As stated in their captions, the proposed delay compensated controller ensures all the conditions proved in Lemma \ref{Prop1}--\ref{lem:interface} with the convergence of the interface position to the setpoint, while the uncompensated control does not provide such a behavior. Hence, the numerical result is consistent with the theoretical result, and the proposed controller achieves better performances than the uncompensated controller under the actuator delay.

\begin{figure}[t]
\centering
\subfloat[Delay compensated control achieves the monotonic convergence of $s(t)$ to the setpoint $s_{{\rm r}}$ without overshooting, i.e.  $\dot{s}(t)>0$, $s_0<s(t)<s_{{\rm r}}$.]{\includegraphics[width=3.0in]{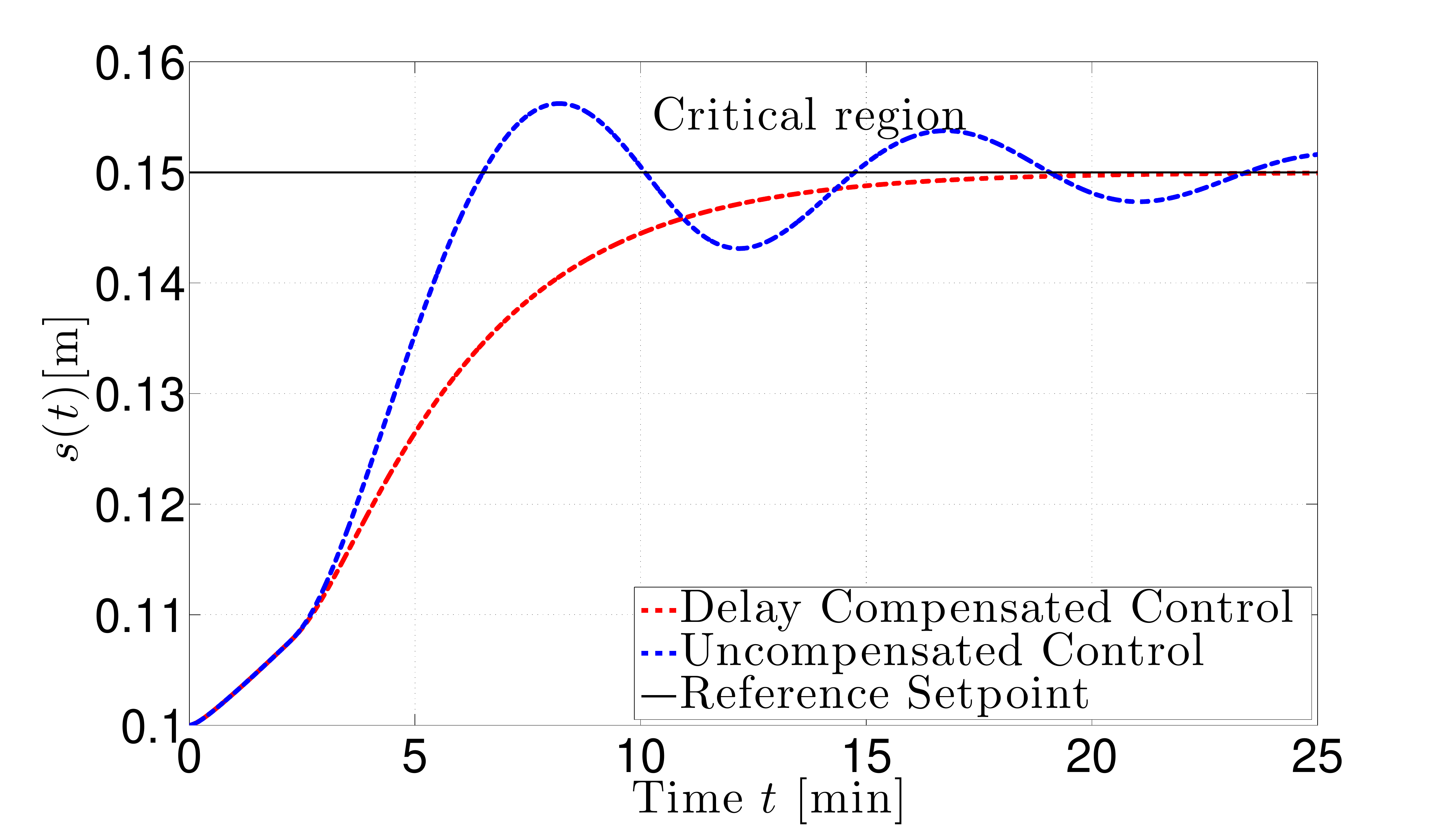}
}
\subfloat[Delay compensated control keeps injecting positive heat, i.e. $q_{c}(t)>0$.]{\hspace{-0.30in}\includegraphics[width=3.2in]{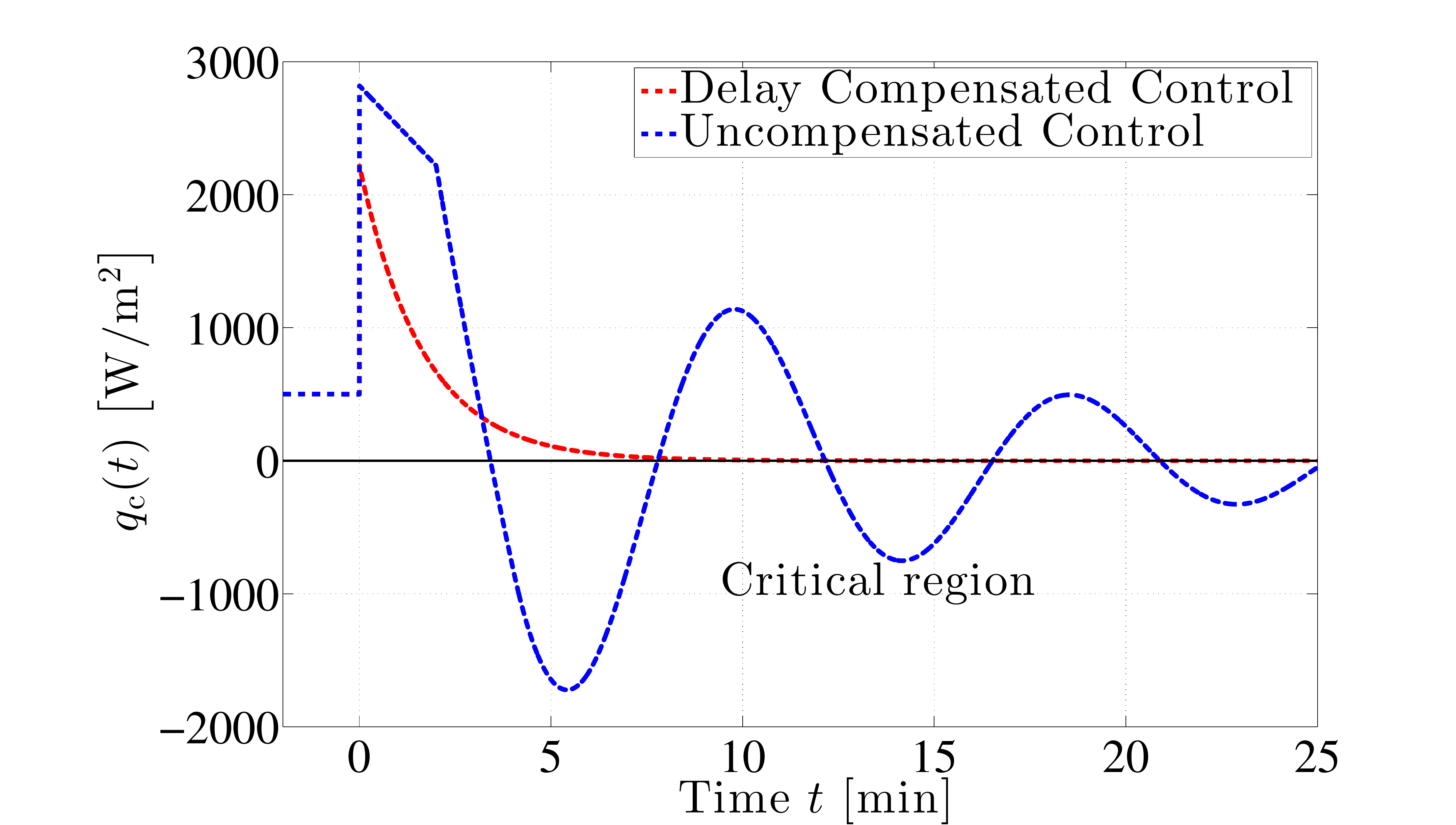}}\\
\subfloat[$T(0,t)$ converges to the melting temperature $T_{{\rm m}}$ with maintaining  $T(0,t)>T_{{\rm m}}$.]{\includegraphics[width=3.0in]{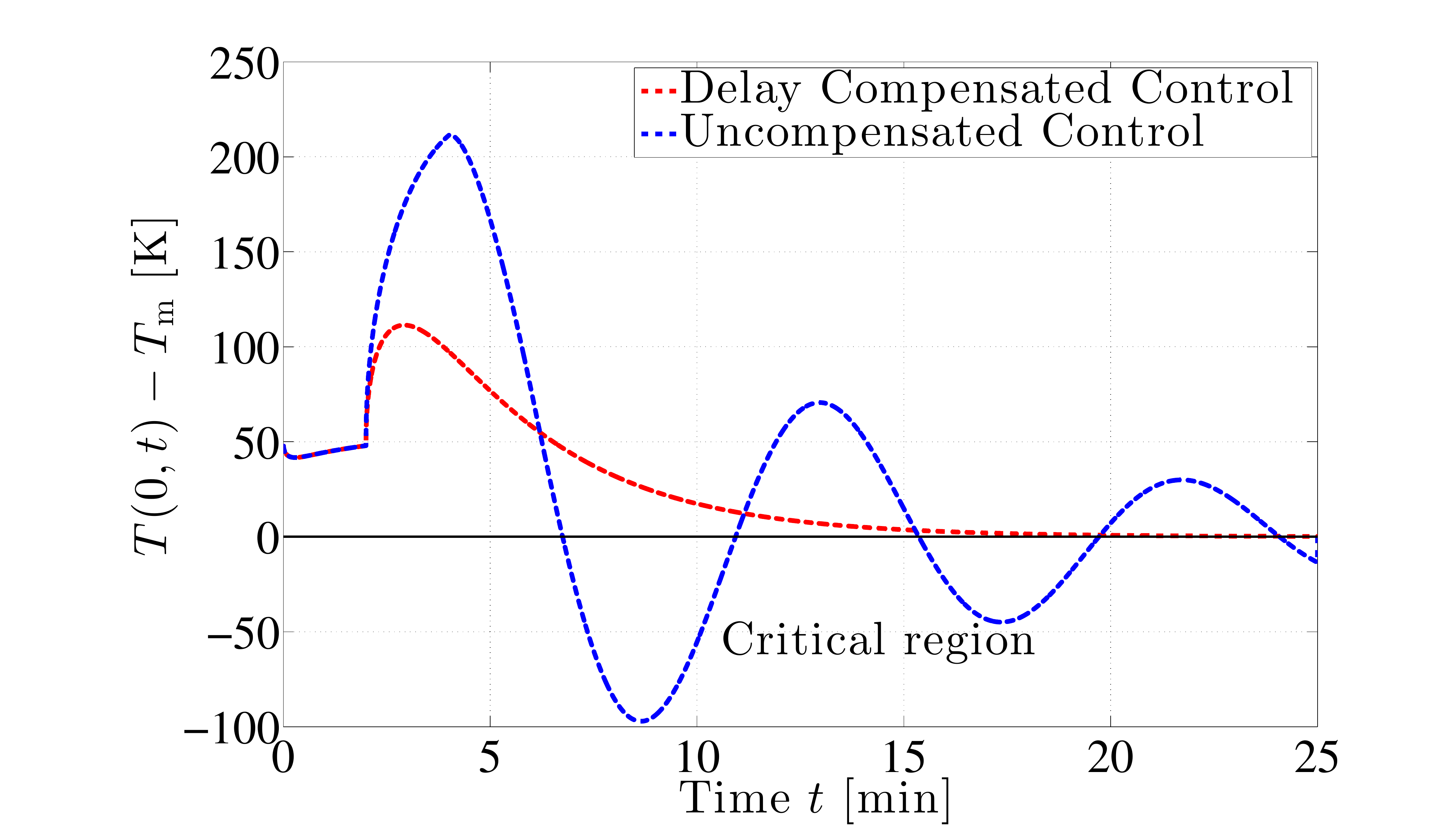}\label{fig:control_1}}
\caption{The closed-loop response of \eqref{eq:stefanPDE}-\eqref{eq:stefanODE} with the delay compensated control law \eqref{delaycomp} (red) and the uncompensated control law \eqref{uncomp} (blue).}
\label{fig:delay}
\end{figure}

\subsection{Robustness to Delay Mismatch} 
To evaluate the delay robustness, the performance of the proposed controller is investigated under the delay mismatch. First, the simulation is conducted with the underestimated delay mismatch where the time delay from the actuator to the plant is 60 [sec] while the compensating time delay in the controller is $D = $ 30 [sec], i.e., the delay mismatch is $\del D = $ 30 [sec]. The closed-loop responses are depicted in Fig. \ref{fig:robust1} with the choices of the control gain $c$ = 0.01 [1/sec] (red) and $c = $ 0.1 [1/sec] (blue). Fig. \ref{fig:robust1} (a) illustrates the convergence of the interface position to the setpoint, however, the monotonicity of the interface dynamics is violated with larger gain (red). From Fig. \ref{fig:robust1} (b) and (c) we can observe that the positivity of the control input and the temperature condition for the liquid phase are satisfied only with the lower gain (blue) for all time, while the simulation with the larger gain (red) violates these conditions too. Hence, with the underestimated delay mismatch, the robustness is well illustrated for sufficiently small gain $c>0$, which is consistent with Theorem \ref{thm:robust}. 

Next, we have studied the simulation with the over-estimated delay mismatch with the same value of the time delay from the actuator to the plant 60 [sec] but the compensating time delay in the controller is $D = $ 90 [sec], i.e., the delay mismatch is $\del D = $ -30 [sec]. The closed-loop responses are depicted in Fig. \ref{fig:robust2} with the same choices of the control gain as in simulation of underestimated delay mismatch. While the magnitude of the delay mismatch is same as the one conducted in the underestimated delay mismatch, we can observe from Fig. \ref{fig:robust2} (b) and (c) that the positivity of the control input and the temperature condition for the model validity are satisfied for all time with both smaller control gain (red) and larger control gain (blue). Although Theorem \ref{thm:robust} guarantees these properties only for sufficiently small control gain $c>0$, the numerical results illustrate that the restriction on the control gain to satisfy these properties is not equivalent between the underestimated and over-estimated delay mismatch. 

Indeed, as far as we have investigated the numerical results with the over-estimated delay mismatch using other values of the control gain $c$ and the delay perturbation $\del D$, the positivity of the control input is satisfied for every cases and the convergence of the interface position to the setpoint is depicted without overshooting. These observations from the numerical simulation leads us to conjecture that the delay-compensated controller might exhibit greater sensitivity to delay mismatch when it is underestimated rather than over-estimated in terms of the model validation. Hence, once the user is faced with some range of the uncertainty in the actuator delay, it is better to choose small control gain $c>0$, and additionally, it might be better to choose larger value of the compensating delay in the controller to be conservative.   


\begin{figure}[t]
\centering
\subfloat[Monotonicity of the interface dynamics is satisfied with smaller gain, but is violated with larger gain.]{\includegraphics[width=3.0in]{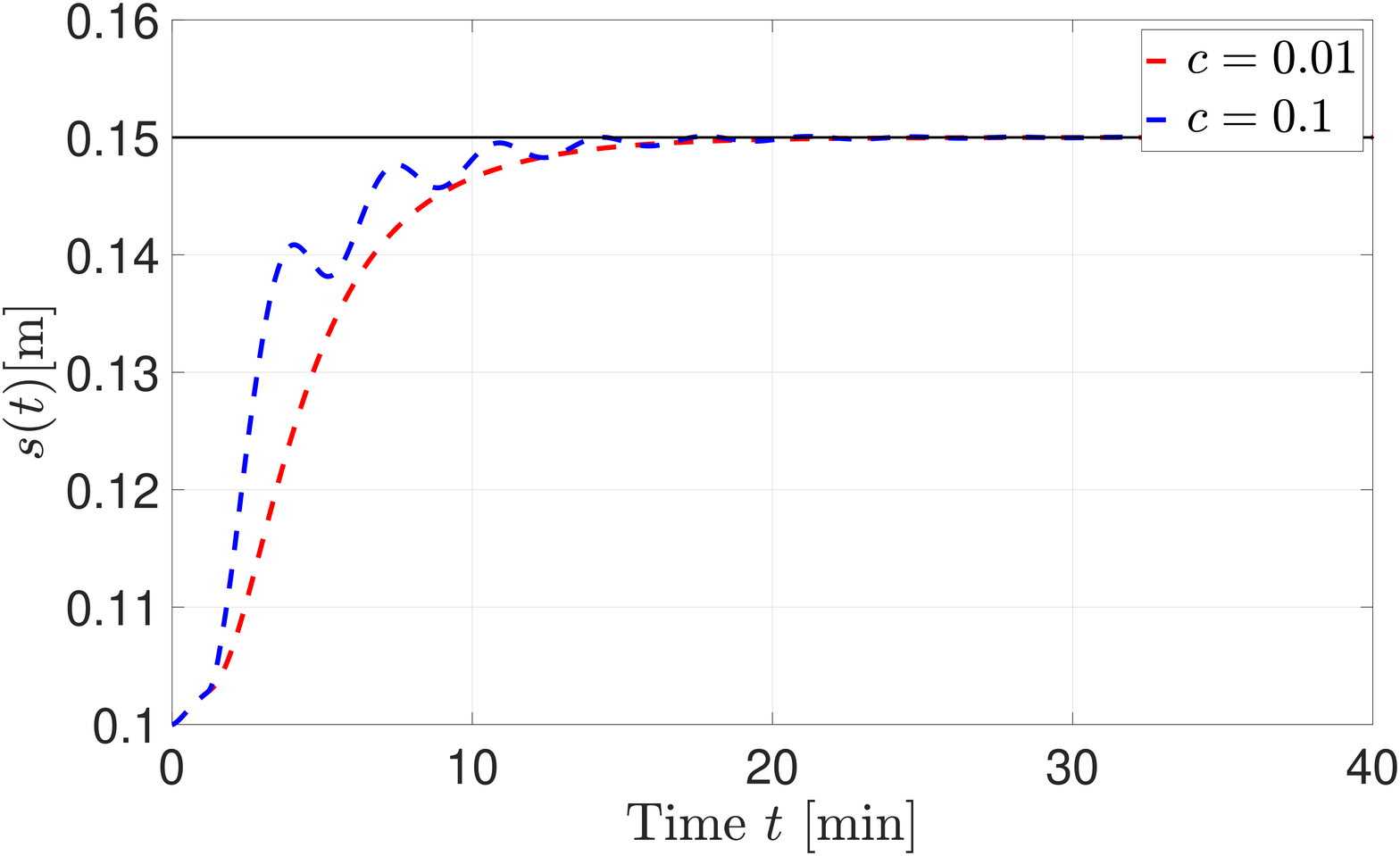}
} 
\subfloat[Positivity of the heat input is satisfied with smaller gain, but is violated with larger gain. ]{\includegraphics[width=3.2in]{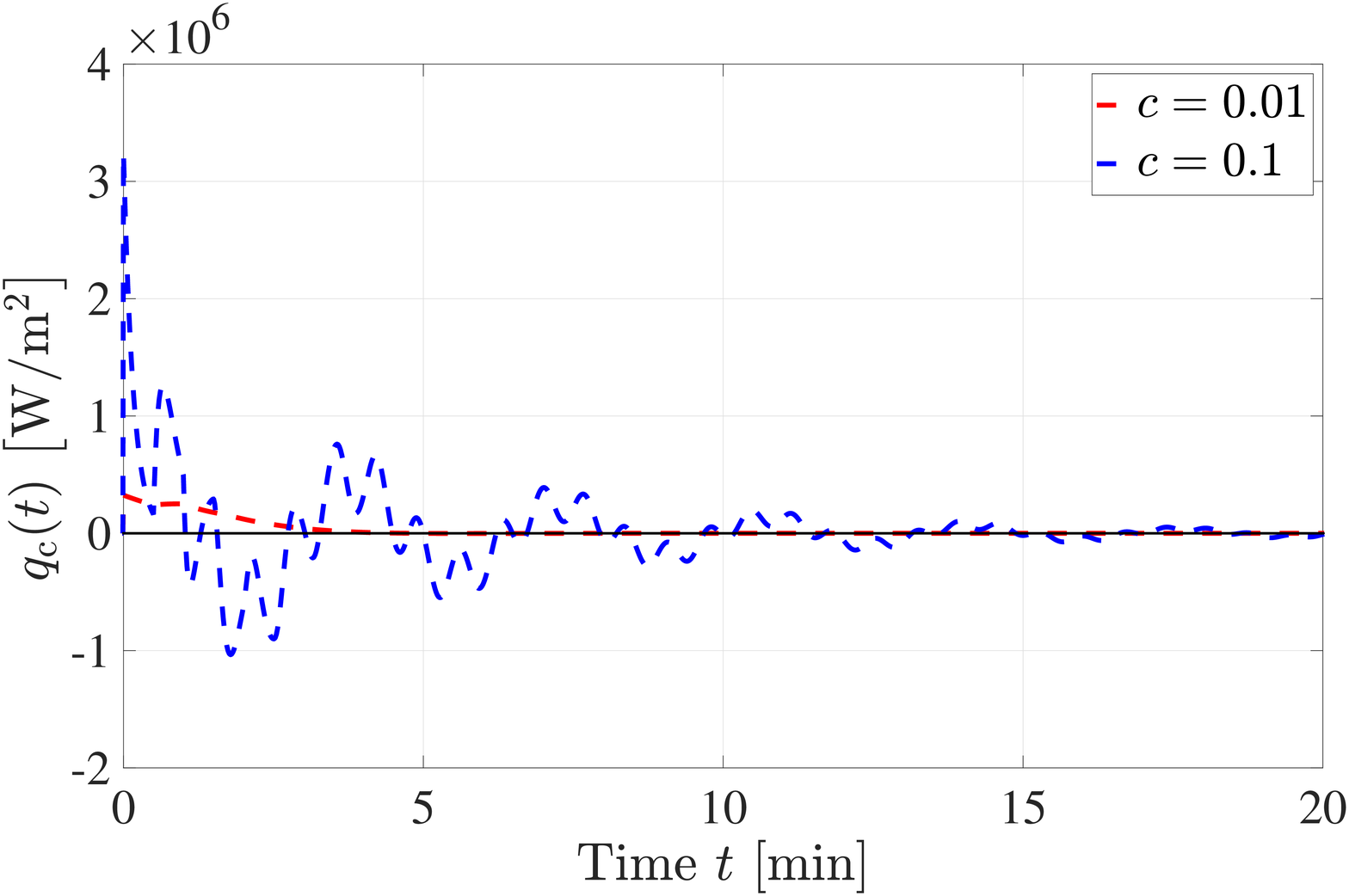}}\\
\subfloat[The boundary temperature keeps above the melting temperature with smaller gain, while it reaches below the melting temperature with larger gain, which violates the temperature condition for the liquid phase. ]{\includegraphics[width=3.0in]{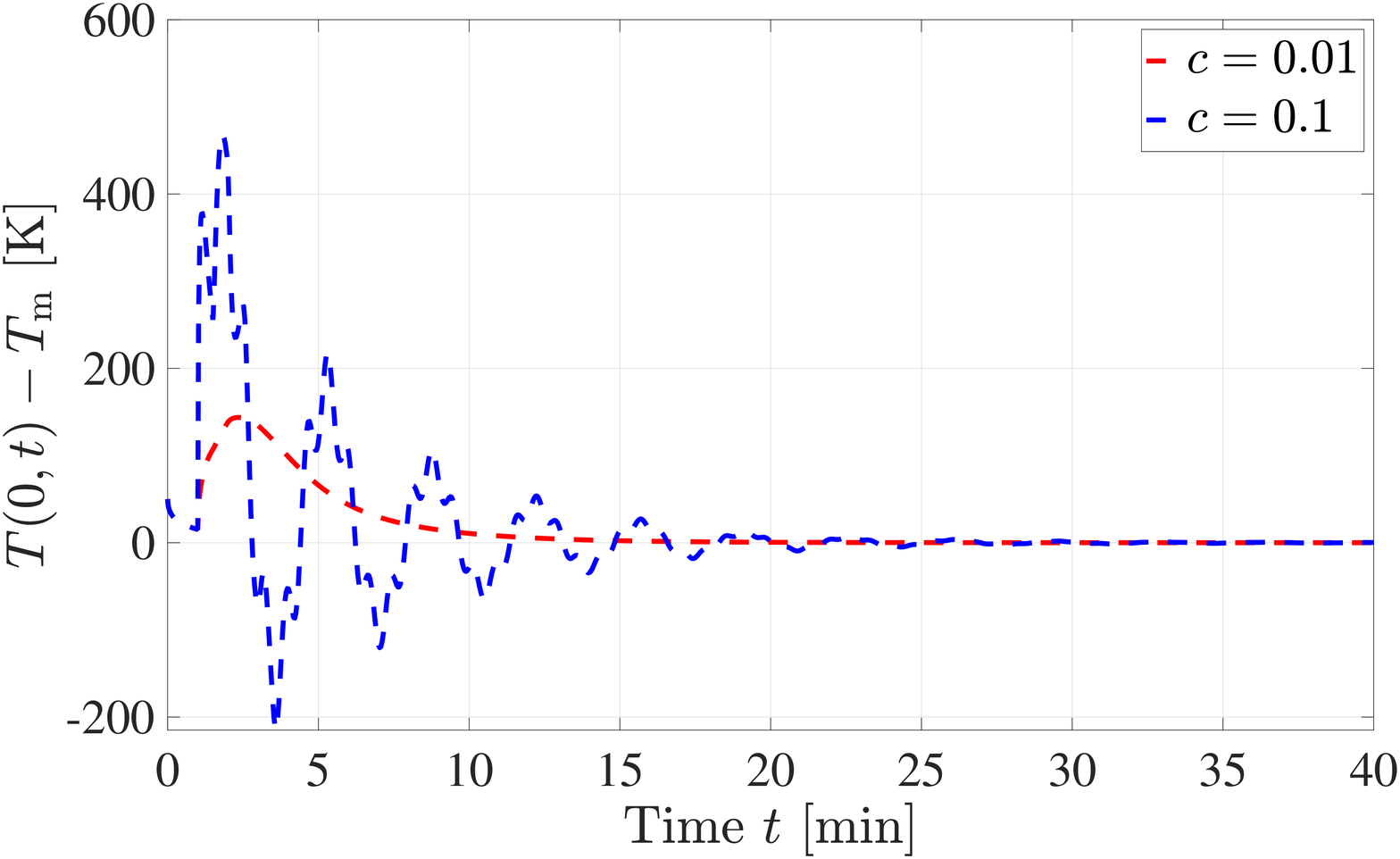}}
\caption{The closed-loop response under the "underestimated" delay mismatch with $D = 30 $ [sec] and $\del D = 30$ [sec]. The simulations are conducted with the control gain $c = 0.01$ [/sec] (red) and $c = 0.1$ [/sec]  (blue). The delay-robustness is observed only with smaller gain in terms of the model validity. }
\label{fig:robust1}
\end{figure} 
\begin{figure}[t] 
\center
\subfloat[The interface position converges to the setpoint without overshooting. ]{\includegraphics[width=3.0in]{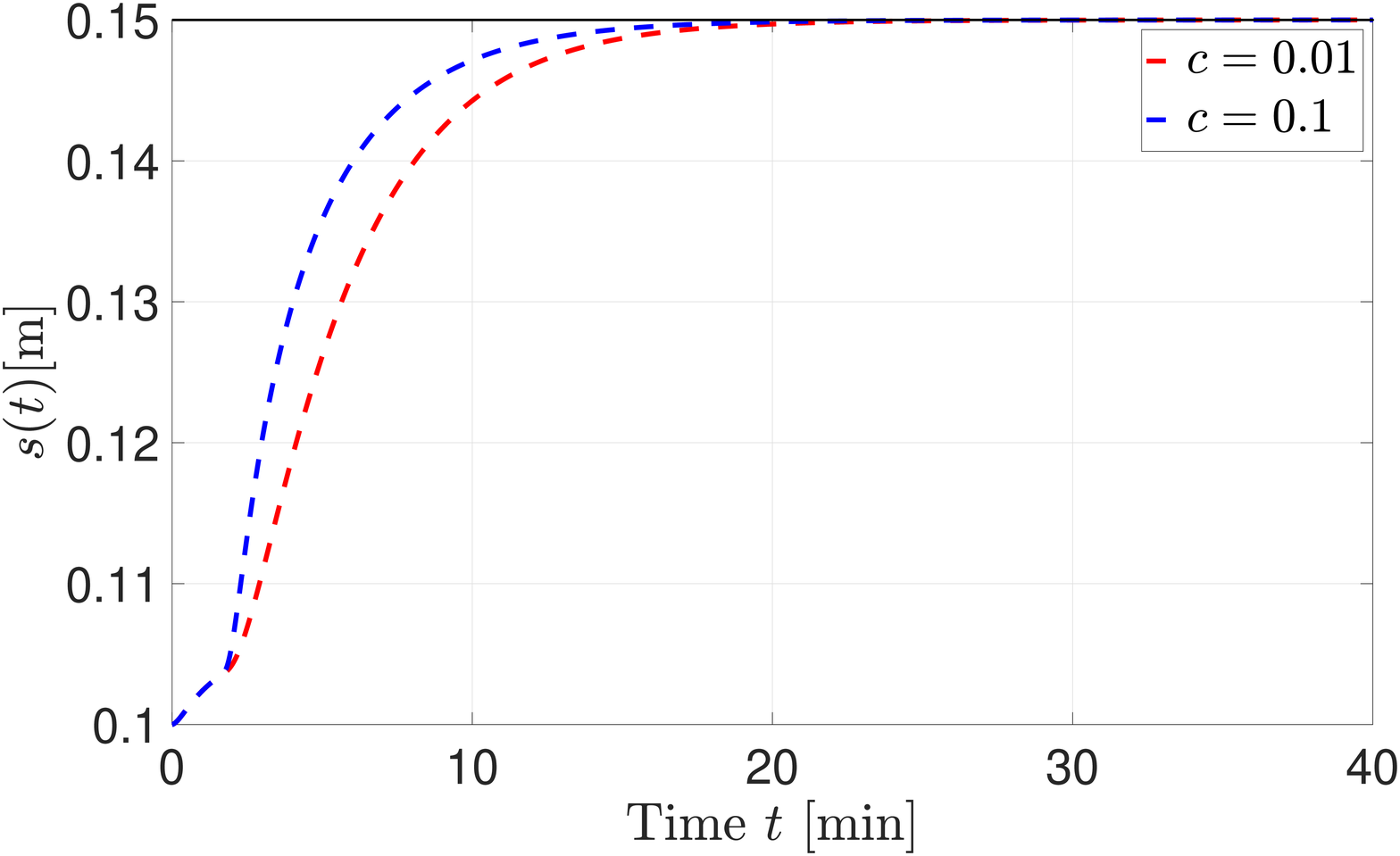}
}
\subfloat[The heat input maintains positive. ]{\includegraphics[width=3.2in]{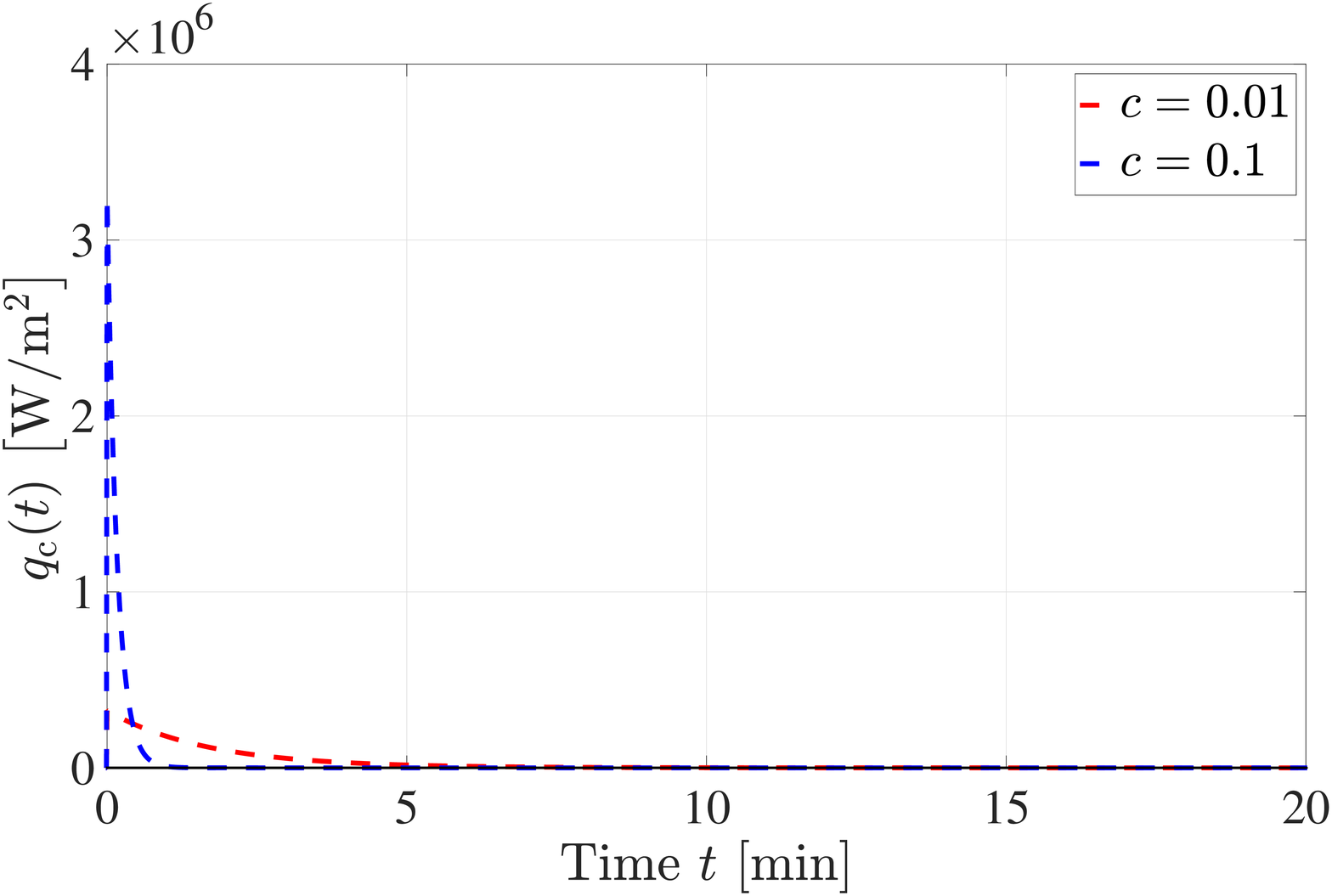}}\\
\subfloat[The boundary temperature is greater than melting temperature, which satisfies the temperature condition for the liquid phase. ]{\includegraphics[width=3.0in]{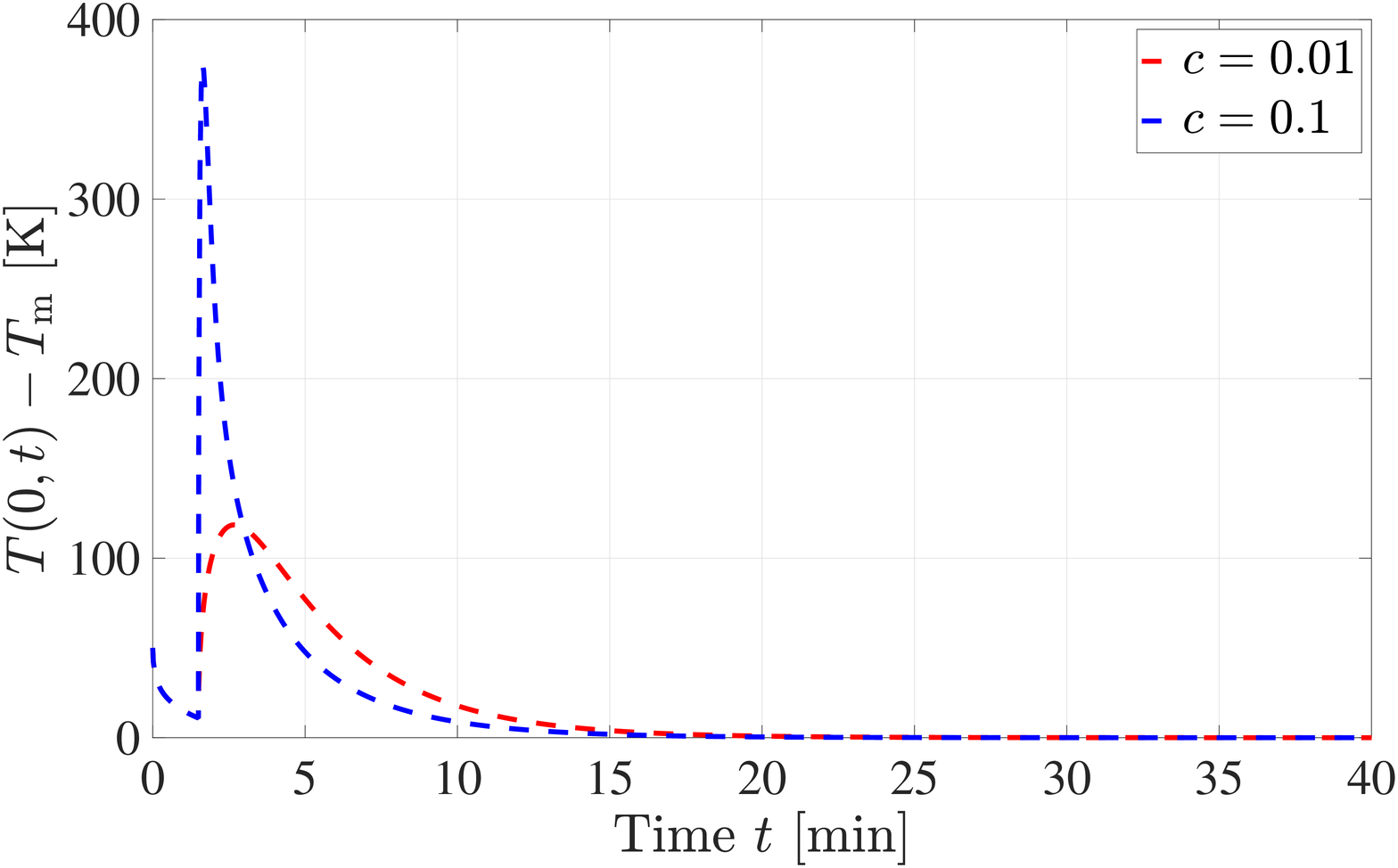}}
\caption{The closed-loop response under the "overestimated" delay mismatch with $D = 90 $ [sec] and $\del D = -30$ [sec]. The simulations are conducted with the control gain $c = 0.01$ [/sec] (red) and $c = 0.1$ [/sec]  (blue). In this case, all the constraints for the model validity are satisfied with both smaller gain and larger gain. }
\label{fig:robust2}
\end{figure}


\section{Conclusions}\label{conclusion}
This paper presented a boundary control design for the one-phase Stefan problem under an actuator delay to achieve exponential stability in the ${\cal H}_1$-norm using full state measurements. Combining our previous contribution \cite{Shumon16} with the contribution in \cite{krstic2009}, two nonlinear backstepping transformations for moving boundary problems are utilized and the associated backstepping controller is proved to remain positive with a proper choice of the setpoint due to energy conservation. Then, some physical constraints required for the validity of the model are verified and the proof of stability. The equivalence to the exact prediction of the nominal control law is shown, and the robustness to delay mismatch is analyzed. 

An analogous state estimation problem is not a trivial extension, especially when the sensor delay appears in the measurement of the interface position. This problem is motivated by the estimation of sea ice melting \cite{Shumon17seaice}, where the thickness of the sea ice can be measured by satellites which causes a time delay to acquire the data through some communication. In such a case, the interface position at current time needs to be estimated, and the estimated interface position should be incorporated in the domain of the estimated temperature profile, which leads to a completely different structure of the system from the plant for control problem we have studied. This state estimation for Stefan problem under a sensor delay is considered as a future work. 


\section*{Acknowledgments}










\appendix


\section{Calculation of double integrals} \label{app:f}
We show \eqref{ftfin} from \eqref{ftdef}. Substituting the inverse transformation \eqref{inv2}, ($v(x,t) = z(x,t) - \int_{x}^{0} \mu (x-y) z(y,t) dy -\frac{\beta}{\alpha} \mu(x) \int_{0}^{s(t)} \zeta(y)  w(y,t) dy - \zeta(s(t)) \mu(x) X(t)$) into \eqref{ftdef} yields 
\begin{align} \label{fapp} 
f(t) =  \int_{- \Delta D}^{ 0} v(x,t) dx = &\int_{- \Delta D}^{ 0} z(x,t) dx -  \int_{- \Delta D}^{ 0}\int_{x}^{0} \mu (x-y) z(y,t) dy dx - \int_{- \Delta D}^{ 0} \mu(x) dx \left( \frac{\beta}{\alpha} \int_{0}^{s(t)} \zeta(y)  w(y,t) dy + \zeta(s(t)) X(t) \right) .
\end{align} 
Since $\mu(x) = c e^{cx}$ (see \eqref{mudef}), the followings are obtained: 
 \ba \label{integmu} 
  \int_{- \Delta D}^{ 0} \mu(x) dx = (1 - e^{ - c \del D}) , 
  \end{align} 
  \ba 
   \int_{- \Delta D}^{ 0}\int_{x}^{0} \mu (x-y) z(y,t) dy dx =& \int_{- \Delta D}^{ 0}\int_{x}^{0} c e^{c (x-y)} z(y,t) dy dx =  e^{cx} \int_{x}^{0} e^{- cy} z(y,t) dy \left|_{x=- \Delta D}^{x=0} +  \int_{- \Delta D}^{ 0} z(x,t) dx   \right.  \notag\\
 = & - e^{- c \Delta D} \int_{- \Delta D}^{0} e^{- cy} z(y,t) dy +  \int_{- \Delta D}^{ 0} z(x,t) dx .  \label{integ2mu} 
 \end{align} 
Substituting \eqref{integmu} and \eqref{integ2mu} into \eqref{fapp}, we arrive at 
 \begin{align} \label{fappfin} 
 f(t) = \int_{- \Delta D}^{ 0} v(x,t) dx  = &   \int_{- \Delta D}^{ 0} e^{- c(x + \Delta D)} z(x,t) dx - ( 1 - e^{ - c \Delta D} ) \left( \frac{\beta}{\alpha} \int_{0}^{s(t)} \zeta(y)  w(y,t) dy + \zeta(s(t)) X(t) \right) , 
  \end{align} 
  which is same as \eqref{ftfin}. The time derivative of \eqref{fappfin} is given with the help of the target system, as 
\begin{align} 
f_{t}(x,t) =  &- \int_{- \Delta D}^{ 0} e^{- c(x + \Delta D)} z_{x}(x,t) dx - ( 1 - e^{ - c \Delta D} ) \left(   \frac{\beta}{\alpha} \int_{0}^{s(t)} \zeta(y)  w_{t}(y,t) dy + \zeta(s(t)) \dot{X}(t) + \dot{s}(t) \zeta'(s(t)) X(t) \right) \notag\\
= & - \left( e^{- c(x + \Delta D)} z(x,t)|_{x= - \Delta D}^{x =0} + c \int_{- \Delta D}^{ 0} e^{- c(x + \Delta D)} z(x,t) dx \right) \notag\\
&- ( 1 - e^{ - c \Delta D} ) \left( z(0,t) - \frac{c\beta}{\alpha}  \int_{0}^{s(t)} \zeta(y)  w(y,t) dy + \dot{s}(t) X(t) \left( \frac{c}{\alpha} \int_{0}^{s(t)} \zeta(y)  dy + \zeta'(s(t)) \right) - c X(t) \zeta(s(t)) \right)
\end{align}
Since $ \zeta(x) = \frac{1}{\beta} {\rm cos}\left( \sqrt{\frac{c}{\alpha}}x\right)$ (see \eqref{mudef}), we have 
 \begin{align} 
f_{t}(x,t) =  & - z(0,t) + z(-\Delta D,t) - c \int_{- \Delta D}^{ 0} e^{- c(x + \Delta D)} z(x,t) dx + ( 1 - e^{ - c \Delta D} ) \left( \frac{c\beta}{\alpha}  \int_{0}^{s(t)} \zeta(y)  w(y,t) dy - c X(t) \zeta(s(t))   \right) \notag\\
= & - \int_{- \Delta D}^{ 0} z_{x}(x,t) dx   - c \int_{- \Delta D}^{ 0} e^{- c(x + \Delta D)} z(x,t) dx + ( 1 - e^{ - c \Delta D} ) \left( \frac{c\beta}{\alpha}  \int_{0}^{s(t)} \zeta(y)  w(y,t) dy - c X(t) \zeta(s(t))   \right) . 
\end{align}

\section{Proof of Lemma 6 } \label{app:lemma} 
By rescaling the time $t$ by defining $\bar t = c t$, $p(\bar t ) = q_{c}(t)$, $\bar D = cD$, $\bar \Delta = c \Delta D$, and by dropping the bar on the variables to reduce the notational burdens, the delay differential equation \eqref{robustcontder} is rewritten as 
 \begin{align} \label{eq:eq-x}
 \dot{p}( t ) =& - p( t ) + p( t - D) - p( t  - D - \Delta  ) , \\
 p_0 =& \psi_{0}  >0.
\end{align} 
Hence, to derive Lemma 6, it suffices to show that there exists $\del ^{**}>0$ such that the solution to \eqref{eq:eq-x} with a positive initial condition is positive. We deduce it by proving the following two lemmas. 

Let us define $T_1 = \min \left\{D, D + \Delta \right\}$, $T_2 = \max \left\{D, D + \Delta \right\}$, $M_p = \max_{s \in [- T_2,0]} |p(s)|$, and $M_p' =  \max_{s \in [- T_2,0]} |\dot{p}(s)|$. 
\begin{lemma}
	\label{lem:halanay}
	There exists $\Delta^*$ such that, if $|\Delta| < \Delta^*$, it holds
	\begin{align}
		|p(t)| \leq M_p e^{-\gamma t} \,, t \geq 0 \label{halanayineq} 
	\end{align}
	for some $\gamma \in (0,1)$. 
\end{lemma}

\begin{proof}
The proof of this lemma follows from Halanay inequality. By Mean-Value theorem, there exists $a(t) \in [t - T_2, t - T_1]$ such that $\Delta \dot p(a(t)) = p( t - D) - p( t  - D - \Delta  ) $ holds, $\forall t\geq 0$. Hence, by the use of such $a(t)$, \eqref{eq:eq-x} is given by 
\begin{align} \label{app:2:1} 
	\dot p(t) =& - p(t) + \Delta \dot p(a(t)) , \quad t\geq 0
\end{align}
Moreover, applying triangle inequality to \eqref{eq:eq-x}, we get 
\begin{align}
	|\dot p(t)| \leq& 3 \max_{s \in [- T_2,0]} p(t+s), \quad t\geq 0. \label{app:2:2} 
\end{align}
Applying \eqref{app:2:2} to \eqref{app:2:1} yields 
\begin{align} \label{app:2:3}
	\dot p(t) \leq& - p(t) + 3 |\Delta| \max_{s \in [-2 T_2,0]} p(t+s),  \quad t\geq 0 . 
\end{align}
Applying Halanay inequality to \eqref{app:2:3} leads to \eqref{halanayineq} with $\Delta^* = \frac{1}{3}$, where $\gamma$ is a solution to 
\begin{align} 
\gamma + 3 |\Delta |e^{\gamma T_2} - 1 = 0. \label{gam_sol} 
\end{align} 
Since the left hand side of \eqref{gam_sol} is a monotonically increasing function in $\gamma$ with having negative value at $\gamma = 0$ and positive value at $\gamma = 1$, the solution to \eqref{gam_sol} satisfies $\gamma \in (0,1)$, which completes the proof of Lemma \ref{lem:halanay}.
\end{proof}

Hereafter we assume $| \Delta | \leq \Delta^* = \frac{1}{3}$. 

\begin{lemma} \label{lem:8} 
	There exists $\Delta^{\star\star} >0$ such that, if $|\Delta| < \Delta^{\star\star}$ and if $\psi_0$ is positive-valued, the solution of~\eqref{eq:eq-x} satisfies
	\begin{align}
		p(t) >& 0 \,, \quad t\geq 0 . \label{eq:lem8}
	\end{align}
\end{lemma}

\begin{proof}
	Assume that the positive lower and upper bounds of $T_1$ and $T_2$ such that $\underline T \leq T_1 \leq T_2 \leq \bar T$ are known. This is verified supposing that $\Delta$ is small enough. 
	
	\textbf{(i) Prove \eqref{eq:lem8} for $t \in [0, T_1]$}:  \\
	Applying the variation of constant formula to \eqref{app:2:1}, it holds 
	\begin{align}
		p(t) =& e^{-t}p(0) + \Delta \int_0^t e^{-(t-s)} \dot p(a(s)) ds, \label{app:b14} 
	\end{align}
	in which $a(s) \in [s-T_2,s-T_1]$. Since $a(s) \in [-T_2, 0]$ for $s \in [0, t]$ and $t \in [0, T_1]$, \eqref{app:b14} leads to 
	\begin{align}
		p(t) \geq & e^{-t}p(0) - |\Delta| M_{p}'\int_0^t e^{-(t-s)}  ds
			= e^{-t}p(0) - |\Delta| M_{p}' (1 - e^{-t}).  \label{app:2:5} 
	\end{align}
	Therefore, choosing $k = M_p - \epsilon >0$ in which $\epsilon>0$ is small enough, and provided that 
	\begin{align}
		|\Delta| \leq \Delta_1 := \frac{\epsilon}{M_{p}' (e^{\overline T} - 1)} , 
	\end{align}
	\eqref{app:2:5} leads to 
	\begin{align}
		\label{eq:ineq-T1}
		p(t) \geq &  e^{-t}p(0) - \epsilon \frac{1 - e^{-t}}{e^{\overline T} - 1}, \notag\\
		\geq & (p(0) - \epsilon)e^{-t}, \notag\\
		\geq & k e^{-t}, \quad t \in [0,T_1]. 
	\end{align}

	\textbf{(ii) Prove \eqref{eq:lem8} for $t \in [T_1, T_2]$}:  \\
	We evaluate $ \Delta \dot p(a(t))$ in \eqref{app:2:1} by separating the cases $a(t) \leq 0$ and $a(t) > 0$ as follows 
	\begin{align}
	\Delta \dot p(a(t)) \leq \left\{ \begin{aligned}
		&	|\Delta| \max_{s\in [-T_2,0]} |\dot p(s)| & \quad \mbox{if } a(t) \leq 0 \\
		&	3 |\Delta| \max_{s \in [-T_2,0]} |p(t+s)| \leq  3 |\Delta| M_{p} & \quad \mbox{if } a(t) > 0
		\end{aligned} \right.
	\end{align}
	where we used Lemma~\ref{lem:halanay} in the second line considering $|p(t) | \leq M_p e^{-\gamma t} \leq M_p$ for $t \geq 0$. Combining the two cases, it holds that $\Delta \dot p(a(t)) \leq | \Delta | M$ where	 $M = \max\left\{\max_{s\in [-T_2,0]} |\dot p(s)|, 3 \max_{s \in [-T_2,0]} |p(s)|  \right\}$. Applying this inequality and the variation of constant formula to \eqref{app:2:1}, as previously, one concludes that
	\begin{align}
		p(t) =& e^{-(t-T_1)}p(T_1) + \int_{T_1}^t e^{-(t-s)} \Delta \dot p(a(s)) ds, \\
			\geq& e^{-t} k - |\Delta| M \int_{T_1}^t e^{-(t-s)} ds =  e^{-t} k - |\Delta| M (1 - e^{-(t-T_1)}),  \label{app:2:6} 
	\end{align}
	in which we used $p(T_1) \geq k e^{-T_1}$ by \eqref{eq:ineq-T1}. Consequently, defining $\tilde \gamma_0 = 1 + \epsilon_0$ for some $\epsilon_0 >0$, and provided that
	\begin{align}
		|\Delta| \leq \Delta_{2} :=  \min\left\{ \Delta_1, \frac{k(1 - e^{-\epsilon_0 \underline T})}{M (e^{\overline T} - e^{\underline T})} \right\}, 
	\end{align}
	one concludes from~ \eqref{app:2:6} that
	\begin{align}
		p(t) \geq k e^{-\tilde \gamma_0 t} \,, \quad t \in [0,T_2], 
	\end{align}
		
	\textbf{(iii) Prove \eqref{eq:lem8} for $t \geq T_2$}:  \\
	Finally, we consider $t\geq T_2$. We define the sequence
	\begin{align}
		t_0 =& T_2 \\ 
		t_{n+1} =& \frac{1}{1-\gamma} \ln \left( 1 + e^{(1-\gamma)t_n} \right) \,, \quad n \in \mathbb{N} \label{app:2:7} 
	\end{align}
	This sequence is increasing as \eqref{app:2:7} leads to $t_{n+1} \geq \frac{1}{1-\gamma} \ln \left( e^{(1-\gamma)t_n} \right) = t_n $. If it were bounded, then the sequence $t_n$ would converge to a point $t^\star$ such that
	\begin{align}
		t^\star =& \frac{1}{1-\gamma} \ln \left( 1 + e^{(1-\gamma) t^\star} \right) \Leftrightarrow e^{(1-\gamma)t^\star}  = 1 + e^{(1-\gamma)t^\star}, 
	\end{align}
	from which we see that such $t^\star$ does not exist. Therefore, the sequence $(t_n)$ is unbounded. Consequently, there exists $n \in \mathbb{N}$ such that $t_n \leq t < t_{n+1}$. Moreover, we define  the sequence 
	\begin{align}
		\tilde \gamma_0 &= 1 + \epsilon_0, \\
		\tilde \gamma_{n+1} &=  1 - \frac{1}{t_n}\ln \left( e^{(1- \tilde \gamma_n)t_n} - C_0 \right) , 
	\end{align}
	where $C_0 \in (0,1)$. By definition, one has $\tilde \gamma_{n+1} > \tilde \gamma_n> 1$ and the following relation
	\begin{align}
		e^{(1-\tilde \gamma_n)t_n} &= e^{(1-\tilde \gamma_{n+1})t_n} + C_0. \label{app:seq2} 
	\end{align}

	Using these sequences, we prove the following statement. \\
	\textbf{(\#)}
	$\forall n \in \N$ it holds $p(t) \geq k e^{-\tilde \gamma_n t}$ for $t \in [0,t_n]$. \\
	The statement (\#) is shown for $n=0$ through (i) and (ii). We use induction approach, namely, assume the statement (\#) is true for $n$, and we prove the statement for $n+1$. It is clear that the statement holds for $t \in [0, t_{n}]$ by the assumption, and therefore, we consider $t \in [t_{n}, t_{n+1}]$. Then, using the variation of constant formula to \eqref{app:2:1} and Lemma~\ref{lem:halanay}, it holds 
	\begin{align}
		p(t) =& e^{-(t-t_n)} p(t_n) + \int_{t_n}^t e^{-(t-s)} \Delta \dot p(a(s)) ds , \notag \\
		\geq & e^{-t} e^{(1-\tilde \gamma_n)t_n} k 
			- 3 |\Delta|  M_{p} \int_{t_n}^t e^{-(t-s)} e^{-\gamma (s-T_2)} ds , \notag \\
		\geq& e^{-t} e^{(1-\tilde \gamma_n)t_n} k 
			- \frac{3 |\Delta| e^{\gamma T_2} M_p}{1-\gamma} ( e^{- \gamma t} - e^{- t + (1-\gamma)t_n}) . \label{appB:pineq} 
	\end{align}
	Then, by the use of \eqref{appB:pineq}, a condition for $p(t) \geq k e^{-\gamma_{n+1}t}$, $t\in [0,t_{n+1}]$ is
	\begin{align}
		e^{(1-\tilde \gamma_n)t_n} k 
			- \frac{3 |\Delta| e^{\gamma T_2} M_p}{1-\gamma} ( e^{(1- \gamma) t} - e^{(1-\gamma)t_n})
		\geq& k e^{(1-\tilde \gamma_{n+1})t}.  \label{app:suff} 
	\end{align}
	With the help of $t \in [t_{n}, t_{n+1}]$, \eqref{app:2:7}, and \eqref{app:seq2}, a sufficient condition on $\Delta$ to satisfy \eqref{app:suff} for $\forall n \in \N$  is 
	\begin{align}
		|\Delta| \leq \Delta_{3} := 
		  \frac{k (1-\gamma) C_0}{3 e^{\gamma T_2} M_{p}} . 
	\end{align}
	Therefore, by construction, if
	\begin{align}
		|\Delta| \leq 
			\min\left\{ \Delta_1,  \Delta_2,  \Delta_3\right\}, 
	\end{align}
	then for any $t\geq 0$, there exists $n \in \mathbb{N}$ such that $t \in [t_n,t_{n+1}]$ and $\tilde \gamma_{n+1} >1$ such that
	\begin{align}
		p(t) \geq k e^{-\tilde \gamma_{n+1}t} > 0, 
	\end{align}
which completes the proof of Lemma \ref{lem:8}. 
\end{proof}

\section{Norm estimate} \label{app:bound} 
\subsection{Bound of $f(t)^2$}\label{app:fbound}
 Applying Young's inequality, we have 
\ba \label{app:ftineq}
 f(t)^2 \leq & 2 \left( \left( \int_{- \Delta D}^{ 0} e^{- c(x + \Delta D)} z(x,t) dx \right)^2  + 2 ( 1 - e^{ - c \Delta D} )^2 \left( \frac{\beta^2}{\alpha^2}\left(  \int_{0}^{s(t)} \zeta(y)  w(y,t) dy \right)^2 + \zeta(s(t))^2 X(t)^2 \right) \right) . 
  \end{align} 
  For both cases of $\del D>0$ and $\del D<0$, by applying Cauchy-Schwarz inequality, we deduce 
  \begin{align} \label{app:ineq1}
  \left( \int_{- \Delta D}^{ 0} e^{- c(x + \Delta D)} z(x,t) dx \right)^2 \leq \left( \int_{\min\{0, \del D\}}^{\max\{0,\del D\}}  e^{-2 c(x + \Delta D)} dx \right) \left( \int_{\min\{0, \del D\}}^{\max\{0,\del D\}} z(x,t)^2 dx \right) = \bar M_1 \left( \int_{\min\{0, \del D\}}^{\max\{0,\del D\}} z(x,t)^2 dx \right), 
  \end{align} 
  where $ \bar M_1 = \frac{{\textrm{sign}} ( \del D)}{2c} \left( 1 - e^{-2c \del D} \right)$. Also, applying Young's, Cauchy-Schwarz inequality, 
  \ba 
  \left(  \int_{0}^{s(t)} \zeta(y)  w(y,t) dy \right)^2 
  & \leq  \left(  \int_{0}^{s(t)} \zeta(y) (\omega(x,t) - \left(x - s(t) \right) z(0,t) ) dy \right)^2 \leq   \left(  \int_{0}^{s(t)} \zeta(y) \omega(x,t) dx +  \frac{1}{\beta} \sqrt{\frac{\alpha}{c}} \left(1 - \cos \left( \sqrt{\frac{c}{\alpha} } s(t) \right) \right) z(0,t)  \right)^2, \notag\\
  \leq & 2 \left( \int_{0}^{s(t)} \zeta(y) \omega(x,t) dx \right)^2 + 2 \frac{\alpha}{\beta^2 c} z(0,t)^2 \leq \frac{ 2s_r}{\beta^2} ||\omega ||^2 + \frac{ 2\alpha}{\beta^2 c} z(0,t)^2. \label{app:ineq2} 
 \end{align} 
Applying \eqref{app:ineq1} and \eqref{app:ineq2} into \eqref{app:ftineq}, the following inequality is derived 
 \begin{align} 
  f(t)^2 \leq & 2 \bar M_1 \left( \int_{\min\{0, -\del D\}}^{\max\{0, -\del D\}} z(x,t)^2 dx \right) + \bar M_2 || \omega||^2 + \bar M_3 z(0,t)^2  + \bar M_4 X(t)^2 , 
  \end{align} 
where 
\begin{align} 
\bar M_2 = \frac{8 s_r}{\alpha^2}  ( 1 - e^{ - c \Delta D} )^2  , \quad \bar M_3 = \frac{ 8}{\alpha c}  ( 1 - e^{ - c \Delta D} )^2 , \quad \bar M_4 =  \frac{4}{\beta^2} ( 1 - e^{ - c \Delta D} )^2. 
\end{align} 

\subsection{Bound of $f'(t)^2$} \label{app:ftbound}

Note that 
 \ba 
f'(t) =  & - \int_{- \Delta D}^{ 0} z_{x}(x,t) dx   - c \int_{- \Delta D}^{ 0} e^{- c(x + \Delta D)} z(x,t) dx + ( 1 - e^{ - c \Delta D} ) \left( \frac{c\beta}{\alpha}  \int_{0}^{s(t)} \zeta(y)  w(y,t) dy - c X(t) \zeta(s(t))   \right) . 
\end{align} 
Thus, applying Young's inequality, 
\begin{align} 
(f'(t))^2 \leq &4 \left( \left( \int_{- \Delta D}^{ 0} z_{x}(x,t) dx \right)^2 + c^2 \left( \int_{- \Delta D}^{ 0} e^{- c(x + \Delta D)} z(x,t) dx \right)^2 + ( 1 - e^{ - c \Delta D} )^2 c^2 \left( \frac{\beta^2}{\alpha^2}  \left(\int_{0}^{s(t)} \zeta(y)  w(y,t) dy\right)^2 +  \zeta(s(t))^2 X(t)^2   \right)  \right) \notag\\
\leq & 4 | \del D| \left( \int_{\min\{0, -\del D\}}^{\max\{0, -\del D\}} z_{x}(x,t)^2 dx \right) + 2 c^2 \bar M_1 \left( \int_{\min\{0, -\del D\}}^{\max\{0, -\del D\}} z(x,t)^2 dx \right)  +  c^2 \left(  \bar M_2 || \omega||^2 + \bar M_3 z(0,t)^2  +  \bar M_4 X(t)^2\right) . 
\end{align}

\nocite{*}
\bibliography{wileyNJD-AMA}%

\clearpage

\section*{Author Biography}


\bibliographystyle{plain}

\end{document}